\documentclass[12pt]{amsart}
\usepackage{amsfonts,graphics,amsmath,amsthm,amsfonts,amscd,amssymb,amsmath, enumerate,latexsym,multicol,mathrsfs}
\usepackage{graphicx, hhline}
\usepackage[all]{xy}
\usepackage[usenames]{color}
\usepackage[colorlinks,
linkcolor=blue,
anchorcolor=black,
citecolor=green
]{hyperref}
\usepackage{fancyhdr}
\usepackage{epsfig,url}
\usepackage[top=30truemm,bottom=30truemm,left=30truemm,right=30truemm]{geometry}

\hypersetup{colorlinks=true}

\title{Log abundance of the moduli b-divisors of lc-trivial fibrations}
\author{Zhengyu Hu}
\date{2021/02/14}

\address{Mathematical Sciences Research Center, 
	Chongqing University of Technology, No.69 Hongguang Avenue, Chongqing, 400054, China}
\email{zhengyuhu16@gmail.com}

\newcommand{\Supp}[0]{{\operatorname{Supp}}}

\DeclareMathOperator{\CDiv}{CDiv}

\DeclareMathOperator{\mult}{mult}

\DeclareMathOperator{\Spec}{Spec}
\DeclareMathOperator{\ex}{Ex}

\newtheorem{thm}{Theorem}[section]

\newtheorem{lem}[thm]{Lemma}
\newtheorem{cor}[thm]{Corollary}
\newtheorem{prop}[thm]{Proposition}

\theoremstyle{definition}
\newtheorem{defn}[thm]{Definition}
\newtheorem{rem}[thm]{Remark}
\newtheorem{note}[thm]{Notation}

\newtheorem{exa}[thm]{Example}
\newtheorem*{ack}{Acknowledgments}

\newtheorem*{claim*}{Claim}

\newcommand{\N}{\mathbb N}
\newcommand{\K}{\mathbb K}
\newcommand{\PP}{\mathbb P}

\newcommand{\Q}{\mathbb Q}
\newcommand{\R}{\mathbb R}
\newcommand{\Z}{\mathbb Z}
\newcommand{\bir}{\dashrightarrow}
\newcommand{\rddown}[1]{\left\lfloor{#1}\right\rfloor} 
\begin{document}

\maketitle

\begin{abstract}
We prove that the moduli b-divisor of an lc-trivial fibration from a log canonical pair is log abundant. The result follows from a theorem on the restriction of the moduli b-divisor, based on a theory of lc-trivial morphisms, which allows us to treat $\R$-divisors and proper morphisms possibly with disconnected fibres. We also prove a theorem on extending a finite cover over a closed subvariety to that over a variety in arbitrary characteristic.
\end{abstract}

\tableofcontents

\section{Introduction}\label{sec1}
We work over an algebraically closed field of characteristic zero unless stated otherwise.\\

{\noindent \textbf{Log abundance of the moduli b-divisor.}}
Given a proper surjective morphism $f:X \to Y$ of normal varieties with connected fibres, and a sub-log canonical (sub-lc for short) sub-pair $(X,B)$ with $K_X+B \sim_\R 0$ over $Y$, there exists a canonical
decomposition of Kodaira type
$$
K_X + B \sim_\R f^*(K_Y + B_Y + M_Y ),
$$
where $B_Y$ and $M_Y$ are $\R$-divisors on $Y$ , called the \emph{discriminant} and
\emph{moduli} $\R$-divisors. For any birational models $X' \to X$ and $Y' \to Y$ such that the induced map $f':X' \to Y'$ is a morphism, we similarly define the discriminant and moduli $\R$-divisors, hence the discriminant $\R$-b-divisor $\mathbf{B}$ and moduli $\R$-Cartier $\R$-b-divisors $\mathbf{M}$ (See Section \ref{sec2-defn} for definition of b-divisors). Using the theory of variations of (mixed) Hodge structure, F. Ambro\cite{ambro1}, and O. Fujino and Y. Gongyo\cite{fujino-gongyo2} show that, if $B$ is a $\Q$-divisor, $f$ has connected fibres, and $\mathrm{rank} f_*\mathcal{O}_X(\lceil  \mathbf{A}^*
(X, B)\rceil) = 1$ (Definition \ref{defn-Q-lc-trivial-fib}), then the moduli b-divisor $\mathbf{M}$ is $\Q$-b-Cartier and b-nef, hence $\mathbf{K}+\mathbf{B}$ is $\Q$-b-Cartier. 

We extend this result to the real coefficients case: 

\begin{thm}[\text{=Theorem \ref{thm--lc-trivial}, cf.\cite[Theorem 2.5]{ambro1}}]\label{thmmain--lc-trivial}
	Notation as above, we have $\mathbf{B}=\sum_i \alpha_i \mathbf{B}_{Y,i}$ is a convex combination of $\Q$-b-divisors $\mathbf{B}_i$ such that $\mathbf{K}+\mathbf{B}_i$ is $\Q$-b-Cartier. In particular, $\mathbf{K}+\mathbf{B}$ is $\R$-b-Cartier. In the same way, the moduli b-divisor $\mathbf{M}=\sum_i \alpha_i \mathbf{M}_{Y,i}$ is a convex combination of $\Q$-b-Cartier and b-nef $\Q$-divisors. 
\end{thm}

As immediate consequences, we prove a Fujino-Mori type canonical bundle formula (Theorem \ref{thm-canonical-bundle-formula}) and the existence of canonical models of Kawamata log terminal pairs (Theorem \ref{thm-canonical-model}). For more details, see \cite{hu2}.

Moreover, if we assume further that $Y$ is complete, and every lc centre is vertical$/Y$, then the moduli b-divisor $\mathbf{M}$ is b-nef and abundant (Corollary \ref{cor--klt-trivial}, cf.\cite{ambro1}\cite{fujino-gongyo2}). 

Using techniques from minimal model theory, O. Fujino and Y. Gongyo \cite{fujino-gongyo2} proved that, the assumption that every lc centre is vertical can be removed. In this paper, we will NOT directly apply the result of \cite{fujino-gongyo2}, but a similar technique will be used in Section \ref{subsec-adj-commute}. Indeed, the main result of \cite{fujino-gongyo2} will be treated uniformly in Section \ref{subsec-log-abundance}.

Generalised (polarised) pairs, introduced by C. Birkar and D. Q. Zhang \cite{birkarzhang}, are originated from the above construction. Notation as above, 
the pair $(X,B+M)$ with data $\mathbf{M}$ is a generalised lc (g-lc for short) generalised pair (g-pair for short) (see Section \ref{sec2-defn} for definition of g-pairs and generalised singularities). We call it the \emph{induced g-pair}. The main purpose of this paper is to study the positivity of the moduli b-divisor $\mathbf{M}$, with respect to the induced g-pair.

Given a g-pair $(X,B+M)$ with data $\overline{M}$, and an $\mathbb{R}$-b-Cartier b-divisor  $\mathbf{D}$ on $X$, we say $\mathbf{D}$ is \emph{b-nef and log abundant with respect to the g-pair} if $\mathbf{D}_{X'}$ is nef and log abundant with respect to any sufficiently high log resolution $(X',B'+M')$ (see Definition \ref{defn-b-nef-log-abundant-divisor}). Note that log abundance is a nice inductive property in birational geometry.

Our main result is the log abundance of the moduli b-divisor:
\begin{thm}[\text{=Theorem \ref{thm-abun-moduli}}]\label{thmmain-abun-moduli}
	Let $f:X \to Y$ be a proper surjective morphism between normal complete varieties with connected fibres. Suppose that $(X,B)$ is an lc pair with $K_X+B \sim_\mathbb{R} 0/Y$. Then the moduli b-divisor $\mathbf{M}$ is b-nef and log abundant with respect to the induced g-pair.
\end{thm}

The connectedness assumption above can be replaced with \emph{goodness} (Definition \ref{defn-good-lc-trivial}) and a modified definition of discriminant b-divisors and moduli b-divisors (Definition \ref{defn-moduli}). In fact, we believe it is more natural to consider ``lc-trivial" type morphism NOT necessarily with connected fibres, since it behaves more flexible in the inductive arguments. Let us explain how we expand the category. 

{\noindent \textbf{Lc-trivial morphisms.}}
Given a proper surjective morphism  $f: (X, B) \to Y$ between normal varieties and a sub-pair $(X, B)$ with $K_X+B \sim_\R 0/Y$, we consider its Stein factorisation $X \overset{\widetilde{f}}{\to} \widetilde{Y} \overset{\gamma}{\to} Y$ and say it \emph{lc-trivial} if $\widetilde{f}:(X,B ) \to \widetilde{Y}$ is lc-trivial (Definition \ref{defn-lc-trivial-morphism}). Moreover, we define the discriminant and moduli b-divisors by proper push-forwards via a finite map (see Definitions \ref{defn-pushdown} and \ref{defn-moduli}). One can verify its positivity by the following theorem. 
Quite recently, J. Han and W. Liu also obtained a similar result \cite[Theorem 4.5]{hanliu}.

\begin{thm}[=Theorem \ref{thm-gen-pair}]\label{thmmain-gen-pair}
	Let $(\widetilde{X}/Z, \widetilde{B} + \widetilde{M})$ be a g-sub-pair with data $\widetilde{\mathbf{M}}$ over a variety $Z$. Let $f : \widetilde{X} \to	X$ be a proper surjective generically finite morphism of normal varieties over $Z$ with $K_{\widetilde{X}} + \widetilde{B} + \widetilde{M} \sim_{\mathbb{K}} 0/X$. Then, there is a g-sub-pair $(X/Z, B + M)$ with data $\mathbf{M}$ such that
	$$
	K_{\widetilde{X}} + \widetilde{B} + \widetilde{M} \sim_{\mathbb{K}}f^*( K_X + B + M)
	$$
	Moreover, if $(\widetilde{X}/Z, \widetilde{B} + \widetilde{M})$ is g-lc (resp. g-klt, g-sub-lc, g-sub-klt), then so is $(X/Z, B + M)$; and if $\widetilde{\mathbf{M}}$ is b-nef and abundant$/Z$ (resp. semi-ample$/Z$), then so is $\mathbf{M}$.
\end{thm}

As to an lc pair, similar operations can be performed to an lc-trivial morphism from an lc pair, so that we can construct a \emph{dlt model} (see Definition \ref{defn-dlt-trivial-fib}). Therefore, the next proposition allows us to study the property of a moduli b-divisor on a dlt model.
\begin{prop}[=Proposition \ref{prop-dlt-trivial-fib}]\label{propmain-dlt-trivial-fib}
		Let $f:(X,B) \to Y$ be an lc-trivial morphism from an lc pair. Then, there exist a birational model $\phi:Y' \to Y$ and a log birational model $(X',B')$ of $(X,B)$ with a B-birational contraction $\pi:(X',B') \dashrightarrow (X,B)$ such that the induced map $f':(X',B')\dashrightarrow Y'$ is a quasi-projective $\Q$-factorial dlt model.
	$$
	\xymatrix{
		X' \ar[d]_{f'} \ar@{-->}[r]^{\pi}   &  X\ar[d]^{f}\  &\\
		Y' \ar[r]^{\phi} &    Y } 
	$$
	In particular, if $f:(X,B) \to Y$ is good, then $f':(X',B') \to  Y'$ is good.
\end{prop}
Let us explain the term \emph{good}. A good lc-trivial morphism means the moduli b-divisors given by $\widetilde{f}$ and $f$, respectively, have a simple relation $\widetilde{\mathbf{M}}=\gamma^* \mathbf{M}$ (Definition \ref{defn-good-lc-trivial}). In particular, a fibration is automatically good. An important observation is that the goodness is preserved under restrictions to strata.

{\noindent \textbf{Adjunction for fibre space commutes with restriction.}}
With the above observation, we are ready to achieve the main technical theorem of this paper. 

The author was informed by E. Floris and V. Lazi\'{c} that they already obtained the result when $B$ is a $\Q$-divisor and $f$ has connected fibres, via a different approach. See \cite[Proposition 4.4]{floris-lazic}. 
\begin{thm}[=Theorem \ref{thm-moduli-div}]\label{thmmain-moduli-div}
	 Let $f:(X,B) \to Y$ be a good dlt model and $T$ be a stratum of the induced dlt pair $(Y,B_Y)$. Suppose $S$ is a stratum of $(X,B)$ saturated over $T$. Then, 
	\begin{enumerate}
		\item $f|_S:(S,B_S) \to T$ is a good dlt model where $K_S+B_S=(K_X+B)|_S$.
		
		\item If we denote by $(T,B_T+M_T)$ the g-dlt pair with data $\mathbf{M}|_T$ given by the adjunction formula $K_T+B_T+M_T=(K_Y+B_Y+M_Y)|_T$, and we denote by $(T,C_T +N_T)$ the g-dlt pair with data $\mathbf{N}$ given by the dlt model $f|_S:(S,B_S) \to T$. Then, $$\mathbf{M}|_T = \mathbf{N}.$$
	\end{enumerate}  
\end{thm}
Combining the above theorem and Proposition \ref{propmain-dlt-trivial-fib}, we can prove Theorem \ref{thmmain-abun-moduli}. \\

{\noindent \textbf{Relative log abundance of the moduli b-divisor.}}
By Lemma \ref{lem-global-local-nef-abun} and the compactification lemma \ref{lem-compact}, we give a relative version of Theorem \ref{thmmain-abun-moduli} as below.
\begin{thm}[=Theorem \ref{thm-relative-log-abundance}]\label{thmmain--relative-log-abundance}
	Let $f:X \to Y$ be a proper surjective morphism between normal varieties with connected fibres, $Y$ be proper over a variety $Z$. Suppose that $(X,B)$ is an lc pair with $K_X+B \sim_\mathbb{R} 0/Y$. Then the moduli b-divisor $\mathbf{M}$ is b-nef and log abundant over $Z$ with respect to the induced g-pair.
\end{thm}
Similarly, the connectedness assumption above can be replaced with goodness.\\

{\noindent \textbf{Extending a finite cover.}} 
We study the extension of a finite cover over a closed subvariety to that over the whole variety. Although only a special case is needed in this paper, assuming the ground field is an algebraically closed field of characteristic zero, we will prove it in greater generality, since it may be interesting to general audience.

\begin{thm}[=Theorem \ref{thm-extend-finite-cover}]\label{thmmain-extend-finite-cover}
	Let $X$ be a normal variety over an arbitrary field and $S$ be a closed subvariety. Suppose we are given a finite morphism $\gamma: \widetilde{S} \to S$ from a normal variety. Suppose further that one of the following conidtions holds:
	\begin{enumerate}
		\item The residue field $\kappa(\eta_S)$ at the generic point $\eta_S$ of $S$ is perfect.
		
		\item $X$ is regular at the generic point of $S$.
	\end{enumerate}
	Then there exists a finite morphism $\rho: \widetilde{X} \to X$ of normal varieties together with a closed subvariety $\widehat{S} \subset \widetilde{X}$ satisfying:
	$$
	\xymatrix{
		\widetilde{S} \ar[dr]_{\gamma}\ar[r]^{\nu} & \widehat{S} \ar[d]^{} \ar@{^(->}[r]^{}   &  \widetilde{X}\ar[d]^{\rho}\  &\\
		& S \ar@{^(->}[r]^{} &    X } 
	$$ 
	
	(1). $\widehat{S}$ is mapped onto $S$ through $\rho$ and the above diagram commutes.
	
	(2). $\nu$ is the normalisation of $\widehat{S}$. 
	
	(3). $\deg \rho= \deg \gamma$. 
\end{thm}

{\noindent \textbf{Sketch of proof of Theorem \ref{thmmain-moduli-div}.}}
First, by induction on dimension, we can assume $S$ is a prime divisor and $T$ is a prime divisor or $Y$. Taking the Stein factorisation and by Proposition \ref{propmain-dlt-trivial-fib}, we can assume $f$ has connected fibres. Replacing $f$ by an appropriate base change by Theorem \ref{thmmain-extend-finite-cover}, we can assume $f|_S$ has connected fibres. Finally we apply weak semi-stable reduction (Theorem \ref{thm-ss-reduction}) to reduce to the case in Lemma \ref{lem-moduli-res}.
\\

{\noindent \textbf{Contents of the paper.}}
In Section \ref{sec2}, we collect definitions, notations and results. 
In Section \ref{subsec-lc-trivial-fib}, we study lc-trivial fibrations with real coefficients, and prove Theorem \ref{thmmain--lc-trivial}.
In Section \ref{subsec-lc-trivial-morphism}, we define (good) lc-trivial morphisms and moduli b-divisors. We also prove Theorem \ref{thmmain-gen-pair}. 
In Section \ref{subsec-dlt-trivial-morphism}, we define dlt models and prove Proposition \ref{propmain-dlt-trivial-fib}. 
In Section \ref{subsec-adj-commute}, we prove Theorem \ref{thmmain-abun-moduli}.
In Section \ref{subsec-log-abundance}, we prove Theorems \ref{thmmain-moduli-div} and \ref{thmmain--relative-log-abundance}.
In Appendix \ref{sec3}, we prove Theorem \ref{thmmain-extend-finite-cover}.
In Appendix \ref{subsec-ssred}, we review a semi-stable reduction with an extra requirement on the base variety. 

\begin{ack}
The author is grateful to Wai-Kit Yeung for discussions and comments on Section \ref{sec3}. He thanks Yifei Chen, Kenta Hashizume, Chen Jiang, Zhan Jiang and Longke Tang for discussions. He also thanks Professors Florin Ambro, Caucher Birkar, Kalle Karu and Aise Johan de Jong 
for answering his questions. Part of the work was carried out when he was visiting CMS, Zhejiang Univeristy. He thanks Professors Kefeng Liu and Hongwei Xu for their hospitality.
\end{ack}

\section{Preliminaries}\label{sec2}
In this section we collect definitions and some important results. Throughout this paper all varieties are over a fixed algebraically closed field of characteristic zero and a divisor refers to an $\R$-Weil divisor unless stated otherwise. 

\subsection{Notations and definitions}\label{sec2-defn}
We collect some notations and definitions. 

\noindent \textbf{Conventions.}
We denote by $\mathbb{K}$ the rational number field $\Q$ or the real number field $\R$. A birational model of a normal variety $X$, often denoted by $X'$, means a variety admits a proper and birational morphism to $X$, and a divisor $D$ over $X$ means a divisor on a birational model of $X$. 

\noindent \textbf{Contractions.}
In this paper a \emph{contraction} refers to a proper morphism $f\colon X\to Y$ of varieties 
such that $f_*\mathcal{O}_X=\mathcal{O}_Y$. In particular, $f$ has connected fibres. Moreover, 
if $X$ is normal, then $Y$ is also normal. A contraction $f$ is \emph{small} if $f$ does not contract any divisor. A birational map $\pi: X \dashrightarrow Y $ is a \emph{birational contraction} if the inverse of $\pi$ does not contract divisors. Note that $\pi$ is not necessarily a morphism unless stated otherwise. 

\noindent  \textbf{Divisors.}
Let $X$ be a normal variety, and let $M$ be a divisor on $X$. 
We denote the coefficient of a prime divisor $D$ in $M$ by $\mult_DM$. 
Writing $M=\sum m_iM_i$ where 
$M_i$ are the distinct irreducible components, the notation $M^{\ge a}$ means 
$\sum_{m_i\ge a} m_iM_i$, that is, we ignore the components with coefficient $<a$. One similarly defines $M^{= a}$. 

By a $\K$-rational function we mean a formal product of finitely many rational functions with $\K$-exponent, namely $\varphi:=\prod_{i=1}^k \varphi_i^{\alpha_i}$ with $\alpha_i \in \K$ for all $i$. We denote its $\K$-Cartier divisor by $(\varphi):=\sum_{i=1}^k \alpha_i (\varphi_i)$. Given two $\R$-Cartier divisors $D,D'$ on $X$, we say $D,D'$ are $\K$-linearly equivalent, and denote by $D \sim_\K D'$, if there exists a $\K$-rational function $\varphi$ such that $D=D'+(\varphi)$.

Given a proper morphism $f:X \to Z$, we say $D,D'$ are $\K$-linearly equivalent (resp. linearly equivalent, equivalent) over $Z$ and denote by $D \sim_\K D'/Z$ (resp. $D\sim D'/Z$, $D=D'/Z$) if there exists an $\R$-Cartier divisor $D_Z$ on $Z$ such that $D \sim_\K D' + f^* D_Z$ (resp. $D \sim  D' + f^* D_Z$, $D= D' + f^* D_Z$).

\noindent \textbf{Very exceptional divisors.}
Let $f : X \rightarrow Y$ be a
dominant morphism from a normal variety to a variety, $D$ a divisor on $X$, and $Z \subset X$ a closed
subset. We say $Z$ is \emph{horizontal} over $Y$ if $f(Z)$ dominates $Y$, and we say $Z$ is \emph{vertical} over $Y$ if $f(Z)$ is a proper subset of $Y$. 

Suppose $f$ is a contraction of normal varieties. Recall that a divisor $D$ is \emph{very exceptional}$/Y$ if $D$ is vertical$/Y$ and for any prime divisor
$P$ on Y there is a prime divisor $Q$ on $X$ which is not a component of $D$ but
$f(Q) = P$, i.e. over the generic point of $P$ we have $\mathrm{Supp} f^\ast P \nsubseteq \mathrm{Supp} D$.

If $\mathrm{codim} f(D) \ge 2$, then $D$ is very exceptional. In this case we say $D$ is \emph{$f$-exceptional}.

For practical reason, we sometimes use the terminology ``very exceptional" even when $f$ is not necessarily a contraction. More precisely, let $f:X \to Y$ be a proper surjective morphism from a normal variety to a variety, and $X \to \widetilde{Y} \to Y$ be the Stein factorisation. We say a divisor $D$ is \emph{very exceptional} over $Y$ if it is very exceptional over $\widetilde{Y}$.

\noindent \textbf{b-divisors.}
We recall some definitions regarding b-divisors. Let $X$ be a variety. A b-divisor $\mathbf{D}$ of $X$ is a family
$\{\mathbf{D}_{X'}\}_{X'}$ of $\R$-Weil divisors indexed by all birational models $X'$ of $X$,
such that $\mu_*(\mathbf{D}_{X''}) = \mathbf{D}_{X'}$ if $\mu: X'' \to X'$ is a birational contraction.

In most cases we focus on a class of b-divisors but not in full generality. An \emph{$\K$-b-Cartier b-divisor} $\mathbf{M}$ is defined by the choice of  
a projective birational morphism 
$\overline{X} \to X$ from a normal variety and an $\K$-Cartier divisor $\overline{M}$ on $\overline{X}$ in the way that $\mathbf{M}_{X'}=\mu^*\overline{M}$ for any birational model $\mu: X' \to \overline{X}$. In this case we say that $\overline{M}$ \emph{represents} $\mathbf{M}$ or $\mathbf{M}$ \emph{descends} to $\overline{X}$. 

Given an $\K$-b-Cartier b-divisors $\mathbf{M}$ on $X$ represented by $M_Y$ and a surjective proper morphism $f: Y \to X$, we define the pull-back of $\mathbf{M}$ as the $\K$-b-Cartier b-divisors $f^*\mathbf{M}$ represented by $\overline{f}^* \overline{M}$ where $\overline{f}: \overline{Y} \to \overline{X}$ is induced by $f$ and $\overline{X} \to X$. 

An $\R$-b-Cartier b-divisor represented by some $\overline{X}\to X$ and $\overline{M}$ is \emph{b-nef} if $\overline{M}$ is 
nef. Similarly we define a \emph{b-nef and abundant} b-divisor if $\overline{M}$ is 
nef and abundant.

We say a few words about the relative case. Let $f : X \to Y$ be a proper surjective morphism of varieties and $\mathbf{D}_1,\mathbf{D}_2$ be $\R$-b-Cartier b-divisors on $X$ represented by $\overline{D_1},\overline{D_2}$ on $\overline{X} \to X$. We say $\mathbf{D}$ is \emph{$\mathbb{K}$-linearly equivalent over $Y$} and denote by $\mathbf{D}_1 \sim_\mathbb{K} \mathbf{D}_2/Y$ if there exists birational models $\phi:X' \to \overline{X}$ and $ Y' \to Y$ such that the induced map $X' \dashrightarrow Y'$ is a morphism and $\phi^* \overline{D_1} \sim_{\mathbb{K}} \phi^* \overline{D_2}/Y'$. The reader may want to check this is well-defined since it is independent of the choice of birational models. In particular, we say $\mathbf{D}$ is \emph{$\mathbb{K}$-linearly trivial over $Y$} if $\mathbf{D} \sim_{\mathbb{K}} 0/Y$. 

\noindent \textbf{Morphisms induced by base change.}
Let $f:X \to Y$ be a dominant morphism from a normal variety to a variety. Given a proper morphism $g: Y' \to Y$ of varieties, we call $f': X \times_Y Y' \to Y'$ the \emph{base change}. If $g$ is surjective, then we consider the normalisation $W:=( X \times_Y Y')^\nu$ which is a disjoint union of finitely many normal varieties, denote by $W=\coprod_i W_i$. For each $i$, we say $W_i$ is a \emph{main component} if $W_i$ dominates both $Z$ and $X$. Recall the basic facts: 
\begin{enumerate}
	\item If either $f$ or $g$ has the connected geometric generic fibre, then the main component is unique. In this case, denoting by $X'$ the unique main component, we say $f':X' \to Y'$ is the morphism \emph{induced by the base change}.
	
	\item If either $f$ or $g$ has equidimensional fibres, then every component is a main component. 
\end{enumerate}

\noindent \textbf{Pairs.} 
A \emph{sub-pair} $(X/Z,B)$ consists of a normal variety $X$, a proper morphism $X \to Z$ and an $\R$-divisor 
$B$ such that $K_X+B$ is $\R$-Cartier. We say $B$ is a \emph{pre-boundary}.
If the coefficients of $B$ are at most $1$ we say $B$ is a 
\emph{sub-boundary}, and if in addition $B\ge 0$, 
we say $B$ is a \emph{boundary}. A sub-pair $(X/Z,B)$ is called a \emph{pair} if $B$ is a boundary.
When $Z$ is not relevant we usually drop it
and do not mention it: in this case one can just assume $X \to Z$ is the identity. 
When $Z$ is a point we also drop it but say the pair is projective.

Let $\phi\colon W\to X$ be a log resolution of a sub-pair $(X,B)$. Let $K_W+B_W$ be the 
pulback of $K_X+B$. The \emph{log discrepancy} of a prime divisor $D$ on $W$ with respect to $(X,B)$ 
is $1-\mult_DB_W$ and it is denoted by $a(D,X,B)$.
We say $(X,B)$ is a \emph{sub-lc pair} (resp. \emph{sub-klt pair}) 
if $a(D,X,B)$ is $\ge 0$ (resp. $>0$) for every $D$. When $(X,B)$ 
is a pair we remove the sub and say the pair is lc, etc. Note that if $(X,B)$ is an lc pair, then 
the coefficients of $B$ necessarily belong to $[0,1]$.  

Let $(X,B)$ be a sub-pair. A \emph{non-klt place} of $(X,B)$ is a prime divisor $D$ on 
birational models of $X$ such that $a(D,X,B)\le 0$. A \emph{non-klt center} is the image on 
$X$ of a non-klt place. When $(X,B)$ is lc, a non-klt center is also called an 
\emph{lc center}. For definitions and standard results on singularities of pairs
we refer to \cite{kollar-mori}.

\noindent \textbf{Generalised pairs.}
For the basic theory of generalised (polarised) pairs we refer to \cite[Section 4]{birkarzhang}.
Below we recall some of the main notions and discuss some basic properties.

A \emph{generalised sub-pair} (\emph{g-sub-pair} for short) consists of 
\begin{itemize}
	\item a normal variety $X$ equipped with a proper
	morphism $X\to Z$, 
	
	\item an $\R$-divisor $B$ on $X$, and 
	
	\item a b-$\R$-Cartier b-divisor $\mathbf{M}$ represented 
	by some projective birational morphism $\overline{X} \overset{\phi}\to X$ and $\R$-Cartier divisor
	$\overline{M}$ on $X$ such that $\overline{M}$ is nef$/Z$ and $K_{X}+B+M$ is $\R$-Cartier,
	where $M := \phi_*\overline{M}$.
\end{itemize}
A generalised sub-pair is a \emph{generalised pair} (\emph{g-pair} or \emph{pair} for short) if $B$ is effective. 

We usually refer to the sub-pair by saying $(X/Z,B+M)$ is a \emph{g-sub-pair with 
data} $\overline{M}$ (or $\mathbf{M}$). Since a b-$\R$-Cartier b-divisor is defined birationally, in practice we will often replace $\overline{X}$ with a log resolution (and hence omit it) and replace $\overline{M}$ with its pullback. In this case, we say $(\overline{X},\overline{B}+\overline{M}) \to X$, where $K_{\overline{X}}+\overline{B}+\overline{M}$ is the pull-back of $K_X+B+M$, is a \emph{data log resolution}.
When $Z$ is not relevant we usually drop it
and do not mention it: in this case one can just assume $X \to Z$ is the identity. 
When $Z$ is a point we also drop it but say the pair is projective. A g-sub-pair naturally defines a b-divisor $\mathbf{K}+\mathbf{B}+\mathbf{M}$ which descends to $X$ so that for any projective birational morphism $X' \overset{\phi'}\to X$ we have $\phi'^*(K_X+B+M)=\mathbf{K}_{X'}+\mathbf{B}_{X'}+\mathbf{M}_{X'}$. We call $\mathbf{B}$ the \emph{pre-boundary b-divisor} and $\mathbf{M}$ the \emph{moduli b-divisor}. Similarly, If the coefficients of $\mathbf{B}$ are at most $1$ we say $\mathbf{B}$ is a 
\emph{sub-boundary b-divisor}, and if in addition $B\ge 0$, 
we say $\mathbf{B}$ is a \emph{boundary b-divisor}.

Two g-sub-pairs $(X/Z,B+M),(X'/Z,B+M)$ with the same data $\overline{M}$ are \emph{B-birational} if there is a common log resolution $X \overset{\pi}{\leftarrow}  \overline{X} \overset{\pi'}{\to} X'$ with $K_{\overline{X}}+\overline{B}+ \overline{M} = \pi^* (K_X+B+M)= \pi'^*(K_{X'}+B'+M')$. In this case, the rational map $(\pi' \circ \pi^{-1}): X \dashrightarrow X'$ is a \emph{B-birational map}. If its inverse map does not contract divisors, then we say it is a \emph{B-birational contraction}. 

Given a g-sub-pair $(X/Z,B+M)$ with data $\overline{M}$, if $\overline{M}$
is a non-negative $\R$-linear combination of $\Q$-Cartier divisors which are
nef over $Z$, then the g-pair $(X/Z, B + M)$ is an \emph{NQC g-pair}. Here, NQC stands for \emph{nef $\Q$-Cartier combinations}.  

\noindent \textbf{Generalised singularities.} 
Now we define generalised singularities of a g-pair.
Replacing $X$ we can assume $\phi$ is a log resolution of $(X,B)$. We can write 
$$
K_{\overline{X}}+\overline{B}+\overline{M}=\phi^*(K_{X}+B+M)
$$
for some uniquely determined $B$. For a prime divisor $D$ on $X$ the \emph{generalised log discrepancy} 
$a(D,X,B+M)$ is defined to be $1-\mult_D\overline{B}$. 

We say $(X,B+M)$ is 
\emph{generalised lc} or \emph{g-lc} (resp. \emph{generalised klt} or \emph{g-klt})
if for each $D$ the generalised log discrepancy $a(D,X,B+M)$ is $\ge 0$ (resp. $>0$).
A g-lc pair $(X/Z,B+M)$ with data $\overline{M}$ is \emph{generalised dlt} or \emph{g-dlt} if $X$ is quasi-projective and there is an open subset $U \subset X$ containing the generic points of all g-lc centres with $(U,B|_U+M|_U)$ log smooth. If in addition $\rddown{B}$ is irreducible, we say the pair is \emph{generalised plt} or \emph{g-plt} for short. Note that \emph{g-sub-lc, g-sub-klt,} etc. can be defined similarly when we drop the assumption of the effectivity of $B$.

A \emph{generalised non-klt center} of a g-sub-pair $(X,B+M)$ is the image of a prime divisor 
$D$ over $X$ with $a(D,X,B+M)\le 0$, and 
the \emph{generalised non-klt locus} of the g-sub-pair is the union of all the generalised non-klt centers. When $(X,B+M)$ is g-lc (resp. g-sub-lc), a generalised non-klt center is also called a 
\emph{g-lc centre} (resp. \emph{g-sub-lc centre}).

\noindent \textbf{Minimal models.}
A g-pair $(Y/Z,B_Y+M_Y)$ with data $M_{\overline{Y}}$ is a \emph{log birational model} of a g-pair $(X/Z,B+M)$ with data $\overline{M}$ if we are given a birational map
$\phi\colon X\bir Y$, $B_Y=B^\sim+E$ where $B^\sim$ is the birational transform of $B$ and
$E$ is the reduced exceptional divisor of $\phi^{-1}$, that is, $E=\sum E_j$ where $E_j$ are the
exceptional/$X$ prime divisors on $Y$ and $\overline{M}_Y=\overline{M}$. 

A log birational model $(\overline{X}/Z,\overline{B}+\overline{M})$ is a \emph{log smooth model} of $(X/Z,B+M)$ if it is log smooth with data $\overline{M}$.

A log birational model $(Y/Z,B_Y+M_Y)$ is a \emph{weak log canonical (weak lc for short) model} of $(X/Z,B+M)$ if

$\bullet$ the data $M_{\overline{Y}}=\overline{M}$,

$\bullet$ $K_Y+B_Y+M_Y$ is nef$/Z$, and

$\bullet$ for any prime divisor $D$ on $X$ which is exceptional/$Y$, we have
$$
a(D,X,B+M)\le a(D,Y,B_Y+M_Y).
$$

A weak lc model $(Y/Z,B_Y+M_Y)$ is a \emph{minimal model} of $(X/Z,B+M)$ if

$\bullet$ $Y$ is $\Q$-factorial,

$\bullet$ the above inequality on log discrepancies is strict.

A minimal model $(Y/Z, B_Y+M_Y)$ is a \emph{log minimal model} if

$\bullet$ $(Y/Z,B_Y+M_Y)$ is g-dlt.

A minimal model $(Y/Z, B_Y+M_Y)$ is a \emph{good minimal model} if $K_Y + B_Y +M_Y$ is semi-ample$/Z$. In this case, $K_Y + B_Y +M_Y$ defines a contraction $g:Y \to W$ such that $K_Y + B_Y +M_Y=g^*A_W$ for some ample$/Z$ divisor $A_W$. We say $W$ is the \emph{canonical model} of $(X/Z,B+M)$.  \\

On the other hand, a log birational model $(Y/Z,B_Y+M_Y)$  is called a \emph{weak Mori fibre space} of $(X/Z,B+M)$ if

$\bullet$ there is a $K_Y+B_Y+M_Y$-negative extremal contraction $Y\to T$
with $\dim Y>\dim T$, and

$\bullet$ for any prime divisor $D$ (on birational models of $X$) we have
$$
a(D,X,B+M)\le a(D,Y,B_Y+M_Y)
$$
and strict inequality holds if $D$ is
on $X$ and contracted$/Y$.

A weak Mori fibre space $(Y/Z,B_Y+M_Y)$ is a \emph{Mori fibre space} of $(X/Z,B+M)$ if

$\bullet$ $(Y/Z,B_Y+M_Y)$ is $\Q$-factorial g-dlt.

\subsection{Nakayama-Zariski decompositions}\label{sec2-decomp}
Nakayama \cite{nakayama} defined a decomposition $D=P_\sigma(D)+N_\sigma(D)$ for any pseudo-effective
$\R$-Cartier divisor $D$ on a smooth projective variety. We refer to this as the \emph{Nakayama-Zariski decomposition}.
We call $P_\sigma$ the positive part and $N_\sigma$ the negative part. We can
extend it to the singular case as follows.
Let $X$ be a normal projective variety and $D$ be a pseudo-effective $\R$-Cartier divisor on $X$. We define $P_\sigma(D)$
by taking a resolution $f\colon W\to X$ and letting $P_\sigma(D):=f_*P_\sigma(f^*D)$. An $\R$-Cartier divisor $D$ is called \emph{pseudo-movable} if $D= P_\sigma(D)$. \\


\noindent \textbf{Asymptotic vanishing orders.}
We collect basic definitions from \cite{nakayama} but in the setting of normal varieties instead of smooth varieties. 

Suppose that $D$ is a divisor on a normal projective variety $X$ with $\kappa_\iota(D) \geq 0$ and $\Gamma$ is a prime divisor over $X$ with its corresponding discrete valuation $\mult_\Gamma$. We define the \emph{asymptotic vanishing order} of $D$ along $\Gamma$ as
$$
o_\Gamma(D):= \inf \{\mult_\Gamma(L)| L \in |E|_\Q \} .
$$
where $0 \le E \sim_\R D$ and $|E|_{\mathbb{K}}$ denotes the sets of effective divisors $L$ satisfying $L\sim_{\mathbb{K}} E$ and the \emph{asymptotic fixed part} as $F(D):=\sum_\Gamma o_\Gamma (D) \Gamma$ where $\Gamma$ runs over all prime divisors on $X$. One can verify the following elementary lemma.
\begin{lem}
	The above definition is independent of the choice of $E$. Moreover, the asymptotic vanishing order coincides with the number derived similarly from $|E|_\R$ instead of $|E|_\Q$.
\end{lem}
\begin{proof}
	Replacing $X$ with a resolution we can assume $X$ is smooth and $\Gamma$ is a prime divisor on $X$. Since the first assertion can be implied by the latter, it suffices to show the latter assertion. Let $E' \sim_\R E$ be another effective divisor and write $E'+ \sum_i a_i(f_i)=E$ where $a_i \in \R$ and $f_i$'s are rational functions. It is enough to show that 
	\begin{equation*}\tag{$\dag$}
	\inf \{\mult_\Gamma(L)| L \in |E|_\Q \} \le \mult_\Gamma E'.
	\end{equation*}
	By \cite[Proof of Lemma 2.5.6]{fujino-book}, for any sufficiently divisible positive integer $n$, there is an injection
	$$
	H^0(X,\mathcal{O}_X(\rddown{nE'})) \hookrightarrow 	H^0(X,\mathcal{O}_X(\rddown{(n+1)E}))
	$$
	given by a rational function $1/g_n$. Moreover, by the construction of $g_n$, we have that $(g_n)$ is a $\Q$-linear combination of $(f_i)$'s and $\lim\limits_{n \to \infty} \frac{1}{n+1}(g_n)= \sum_i a_i(f_i)$. Now set $E_n=E+\frac{1}{n+1}(g_n)$. We obtain $0 \le E_n \sim_\Q E$ and $\lim\limits_{n \to \infty} \mult_\Gamma E_n =\mult_\Gamma E'$ which in turn implies the inequality $(\dag)$. 
\end{proof}
Note that when $D$ is a $\Q$-Cartier divisor, the definition here coincides with the classical definition due to the previous lemma and a standard argument from convex geometry. 

Let us fix an ample divisor $A$. We define the \emph{numerical vanishing order} of a pseudo-effective divisor $D$ along $\Gamma$ as
$$
\sigma_\Gamma(D):= \lim\limits_{\epsilon \downarrow 0} o_\Gamma(D+\epsilon A).
$$
\cite[III, 1.7. Lemma]{nakayama} verifies that the definition above is independent of the choice of $A$. Readers may notice that the variety here is not required to be smooth. Since we work with discrete valuations, by passing to a resolution, the argument from \cite{nakayama} still works. Moreover, $\sigma_{\Gamma}$ defines a lower semi-continuous convex function on the pseudo-effective cone $\mathrm{PE}(X) \subset N^1(X)$, and $\sigma_\Gamma=o_\Gamma$ is continuous on the big cone $\mathrm{Big}(X)$. Hence we have $\sigma_\Gamma(D):= \lim\limits_{\epsilon \downarrow 0} \sigma_\Gamma(D+\epsilon E)$ for any pseudo-effective divisor $E$.

By an easy calculation one deduces that the negative part of the Nakayama-Zariski decomposition of $D$ has the coefficients which are numerical asymptotic vanishing orders: $N_\sigma(D) =\sum_\Gamma \sigma_\Gamma(D) \Gamma$, where $\Gamma$ runs over all prime divisors on $X$. Note that in general neither $N_\sigma(D)$ nor $P_\sigma(D)$ are $\R$-Cartier.

\subsection{Nef and abundant divisors}\label{sec2-dim}
In this subsection we will introduce the notion of nef and abundant divisor and elementary properties in the setting of $\R$-divisors. Most contents of this subsection are taken from \cite{nakayama} with slight modifications. We write proofs of some results for the reader's convenience.

\noindent \textbf{Iitaka dimension and numerical dimension.}
Recall the following definitions of Iitaka dimension and numerical dimension. Both integers are birational invariants given by the growth of the quantity of sections. 

\begin{defn}[Very general fibre]
	Let $f:X \to Y$ be a contraction of normal varieties, $k \subseteq K$ be a field extension to an uncountable field and $f_K: X_K \to Y_K$ be the morphism induced by the base change $\Spec K \to \Spec k$. We say $F$ is a \emph{very general fibre} of $f$ with respect to $k\subseteq K$ if the fibre is over a very general point, i.e. a point outside a countably many proper closed subsets of $Y_K$. When $K=k$ or $K$ is not relevant we usually do not mention it.
\end{defn}

\begin{defn}[Invariant Iitaka dimension]\label{defn--inv-iitaka-dim}
	Let $X$ be a normal complete variety, and $D$ be an $\mathbb{R}$-Cartier divisor $D$ on $X$. 
	We define the {\em invariant  Iitaka dimension} of $D$, denoted by $\kappa_{\iota}(X,D)$, as follows (see also \cite[Definition 2.5.5]{fujino-book}):  
	If there is an $\mathbb{R}$-divisor $E\geq 0$ such that $D\sim_{\mathbb{R}}E$, set $\kappa_{\iota}(X,D)=\kappa(X,E)$. 
	Here, the right hand side is the usual Iitaka dimension of $E$. 
	Otherwise, we set $\kappa_{\iota}(X,D)=-\infty$. 
	We can check that $\kappa_{\iota}(X,D)$ is well-defined, i.e., when there is $E\geq 0$ such that $D\sim_{\mathbb{R}}E$, the invariant Iitaka dimension $\kappa_{\iota}(X,D)$ does not depend on the choice of $E$. 
	By definition, we have $\kappa_{\iota}(X,D)\geq0$ if and only if $D$ is $\mathbb{R}$-linearly equivalent to an effective $\mathbb{R}$-divisor. 
	
	Let $X\to Z$ be a proper morphism from a normal variety to a variety, and let $D$ be an $\mathbb{R}$-Cartier divisor on $X$. 
	Then the {\em relative invariant Iitaka dimension} of $D$, denoted by $\kappa_{\iota}(X/Z,D)$, is defined by $\kappa_{\iota}(X/Z,D)=\kappa_{\iota}(X,D|_{F})$, where $F$ is a very general fibre of the Stein factorisation of $X\to Z$.
	Note that the value $\kappa_{\iota}(X,D|_{F})$ does not depend on the choice of $F$ (see \cite[Lemma 2.10]{hashizumehu}). 
\end{defn}

\begin{defn}[Numerical dimension]\label{defn--num-dim}
	Let $X$ be a normal projective variety, and $D$ be an $\mathbb{R}$-Cartier divisor $D$ on $X$. 
	We define the {\em numerical dimension} of $D$, denoted by $\kappa_{\sigma}(X,D)$, as follows (see also \cite[V, 2.5 Definition]{nakayama}): 
	For any Cartier divisor $A$ on $X$, we set
	$$
	\sigma(D;A)={\rm max}\left\{k\in \mathbb{Z}_{\geq0}\middle|\, \underset{m\to \infty}{\rm lim}{\rm sup}\frac{{\rm dim}H^{0}(X,\mathcal{O}_{X}(\llcorner mD \lrcorner+A))}{m^{k}}>0\right\}
	$$
	if ${\rm dim}H^{0}(X,\mathcal{O}_{X}(\llcorner mD \lrcorner+A))>0$ for infinitely many $m\in \mathbb{Z}_{>0}$, and otherwise we set $\sigma(D;A):=-\infty$. 
	Then, we define 
	$$\kappa_{\sigma}(X,D):={\rm max}\{\sigma(D;A)\,|\,A{\rm\; is\; a\;Cartier\;divisor\;on\;}X\}.$$
	
	Let $X\to Z$ be a projective morphism from a normal variety to a variety, and let $D$ be an $\mathbb{R}$-Cartier divisor on $X$. 
	Then, the {\em relative numerical dimension} of $D$ over $Z$, denoted by $\kappa_{\sigma}(X/Z,D)$, is defined by $\kappa_{\sigma}(F,D|_{F})$, where $F$ is a very general fibre of the Stein factorisation of $X\to Z$.  
	We note that the value $\kappa_{\sigma}(F,D|_{F})$ does not depend on the choice of $F$, so the relative numerical dimension is well-defined. 
\end{defn}

For a collection of basic properties of the invariant Iitaka dimension and the numerical dimension, we refer to \cite[Remark 2.8]{hashizumehu}. 

\begin{defn}[Relatively abundant divisor and relatively log abundant divisor]\label{defn--abund}
	Let $f\colon X\to Z$ be a projective morphism from a normal variety to a variety, and $D$ be an $\mathbb{R}$-Cartier divisor on $X$. 
	We say that $D$ is {\em abundant over} $Z$ if the equality $\kappa_{\iota}(X/Z,D)=\kappa_{\sigma}(X/Z,D)$ holds. 
	When $Z$ is a point, we simply say $D$ is {\em abundant}. 
	
	Let $f\colon X\to Z$ and $D$ be as above, and $(X,B+M)$ be a g-sub-lc sub-pair. 
	We say that $D$ is $\pi$-{\em log abundant} (or {\em log abundant over} $Z$) with respect to $(X,B+M)$ if $D$ is abundant over $Z$ and for any g-lc center $S$ of $(X,B+M)$ with the normalization $S^{\nu}\to S$, the pullback $D|_{S^{\nu}}$ is abundant over $Z$. 
\end{defn}

\begin{rem}\label{rem--abundant-generic}
	With notation in Definition \ref{defn--abund}, let $X_{\overline{\eta}}$ be the geometric generic fibre of the Stein factorisation of $f$. 
	By flat base change, one can easily verify that $D$ is abundant over $Z$ if and only if the restriction $D|_{X_{\overline{\eta}}}$ is abundant. Moreover, if we denote by $X_{\overline{\eta_i}}$ the geometric generic fibres of the Stein factorisation of $f|_{S_i^\nu}$ for each g-sub-lc centre $S_i$, then, $D$ is log abundant over $Z$ if and only if $D|_{X_{\overline{\eta}}}$ is abundant and $D|_{X_{\overline{\eta_i}}}$ is abundant for every $i$. 
\end{rem}

\noindent \textbf{Nef and abundant divisors.}
Recall that, given a a proper morphism $f\colon X\to Z$ from a normal variety to a variety, an $\R$-Cartier divisor $D$ is \emph{semi-ample} over $Z$ if there exist a proper surjective morphism $h: X \to Y$ over $Z$ and an ample divisor $D_Y$ of $Y$ such that $D \sim_\R h^*D_Y$.

\begin{lem}\label{lem-semi-ample-divisor}
	Notation as above, let $D$ be an $\R$-Cartier divisor. 
	\begin{enumerate}
		\item $D$ is semi-ample$/Z$ if and only if $D$ is a convex combination of semi-ample$/Z$ $\Q$-divisors.
		
		\item Let $D'$ be another $\R$-Cartier divisor. If $D,D'$ are semi-ample$/Z$, then so is $D+D'$.
		
		\item Let $g:W \to X$ be a proper surjective morphism. Then, $D$ is semi-ample$/Z$ if and only if $g^*D$ is semi-ample$/Z$.
	\end{enumerate}
	 
\end{lem}
\begin{proof}
	(1) can be proved by an elementary argument from convex geometry. (2) is a direct consequence of (1). Note that (3) is obvious when $f$ is a contraction and $D$ is $\Q$-Cartier. Suppose $g$ is a finite morphism and $D$ is $\Q$-Cartier. Since $f$ is finite flat over an open subset of $Y$ outside a locus of codimension $\ge 2$, one can easily verify (3) by $g_*g^*D=(\deg g)D$ (see \cite[\href{https://stacks.math.columbia.edu/tag/02RH}{Lemma 02RH}]{stacks-project}). By the Stein factorisation, it remains to verify the case $D$ is not $\Q$-Cartier, which follows from (1).
\end{proof}

\begin{lem}[Nef and abundant divisor, \text{cf.\cite[V. 2.3 Lemma]{nakayama}}]\label{lem-nef-abundant-divisor}
	Let $f\colon X\to Z$ be a projective morphism from a normal variety to a variety, and let $D$ be a nef$/Z$ $\R$-Cartier divisor on $X$. Then, the following conditions are equivalent:
	\begin{enumerate}
		\item $D$ is abundant over $Z$.
		
		\item There exist a birational model $\pi: X' \to X$, a surjective morphism $g : X' \to Y$ of smooth quasi-projective varieties over $Z$,
		and a nef and big$/Z$ divisor $B$ of $Y$ such that $\pi^*D \sim_\R g^*B$.
		
		\item Given a sufficiently general fibre $F$, the asymptotic vanishing order $o_\Gamma(D|_F)=0$ for every prime divisor $\Gamma$ over $F$.
	\end{enumerate} 
\end{lem}
\begin{proof}
	First we prove (1) implies (2). By \cite[V. 2.3 Lemma(1)]{nakayama} and Remark \ref{rem--abundant-generic}, there exist a non-empty open subset $U_Z \subset Z$,  a birational model $\pi: X' \to X$, a surjective morphism $g : X' \to Y$ of smooth quasi-projective varieties over $Z$,
	and an $\R$-Cartier divisor $B$ of $Y$ such that $(\pi|_{U'})^*(D|_{U}) \sim_\R (g|_{U'})^*(B|_{U_Y})$, where $U,U'$ and $U_Y$ are the inverse images of $U_Z$, and $B|_{U_Y}$ is nef and big over $Z$. Applying the equidimensional reduction, there exist birational models $\pi':X'' \to X', \phi:Y' \to Y'$ so that the induced morphism $g': X'' \to Y'$ is equidimensional. Let $\varphi$ be the $\R$-rational function such that $\pi^*D +(\varphi) -g^*B =0$ over $U$. Since $g'$ is equidimensional, by an easy argument (for example, see \cite[Proof of Lemma 3.9]{hu}), there exists a divisor $\Delta$ on $Y'$ so that $(\pi \circ \pi')^*D +(\varphi) =g'^*(\phi^*B+\Delta) +E$ where $E \ge 0$ is very exceptional$/Y'$. Hence, by the negativity lemma \cite[Lemma 3.3]{birkar-flip}, we deduce $E=0$. Replacing $Y$ with $Y'$, $B$ with $\phi^*B+\Delta$ and $X$ with a resolution of $X''$, we complete the proof.
	
	The implication from (2) to (3) is obvious. Now we prove (3) implies (1). By Definition \ref{defn--abund}, replacing $X$ with $F$, we may assume $Z$ is a point. By definition of asymptotic vanishing order, replacing $D$ we may assume $D \ge 0$. Take an Iitaka fibration $g:X' \to Y$. By \cite[V. 2.3 Proof of Lemma(1)]{nakayama}, we may assume $g$ is equidimensional and $\pi^*D \sim_\Q g^*B +E$ where $\pi:X' \to X$ and $B$ is a big divisor. Since $o_\Gamma(D)=0$ for every prime divisor $\Gamma$ over $X$, and $\kappa(E/Y)=0$, we deduce $E$ is vertical$/Y$ and hence $E=0/Y$ by the negativity lemma \cite[Lemma 3.3]{birkar-flip}. The rest follows directly from \cite[V. 2.7 Proposition]{nakayama}. 
\end{proof}

\begin{rem}
	One can regard (2) as an alternative definition of nef and abundant $\R$-Cartier divisor (see \cite{ambro2}\cite{fg-bundle}). Note that by our definition of abundant divisor, notation as above, one cannot expect $\pi^*D \sim_\Q g^*B$ as in \cite[V. 2.3 Lemma(1)]{nakayama} unless $D \sim_\Q E/U_Z$ for some divisor $E \ge 0$ and a non-empty open subset $U_Z \subset Z$ (by \cite[Lemma 2.10]{hashizumehu}). Nevertheless, if $D\ge 0$ is effective, then we can find $B$ so that $\pi^*D = g^*B$.
\end{rem}

\begin{lem}\label{lem-nef-abund-1}
	Let $f:X \to Z$ be a proper morphism from a normal variety to a variety, $g:Y \to X$ be a proper surjective morphism of normal varieties, and $D$ be an $\R$-Cartier divisor on $X$. Then, $D$ is nef and abundant over $Z$ if and only if $f^*D$ is nef and abundant over $Z$.
\end{lem}
\begin{proof}
	The lemma follows directly from \cite[V. 2.7. Proposition]{nakayama}.
\end{proof}

\begin{lem}\label{lem-descend}
	Let $f:X \to Y$ be a surjective finite morphism and $D$ be an $\R$-Cartier divisor on $X$. Let $\phi:Y' \to Y$, $\pi:X' \to X$ be birational models and $f':X' \to Y'$ be a morphism such that $ \phi \circ f'= f \circ \pi$. Suppose $f_*D$ is $\R$-Cartier. Then, we have 
	$$
	f'_* \pi^* D= \phi^* f_* D.
	$$
\end{lem}
\begin{proof}
	If $\phi$ is a small contraction, then there is nothing to prove. Suppose the set of prime exceptional$/Y$ divisors is non-empty, and we pick an element $E_Y$. Let $\sum_k E_{k}$ be a reduced divisor whose components are all prime divisors mapped onto $E_Y$. It suffices to show the equation near the point $y$. 
	Since the question is local, by shrinking $Y$ and the others accordingly, we may assume $X,Y$ are affine. 
	$$
	\begin{array}{cclcl} 
	\sum_k E_{k} & \subset & X' &\stackrel{\pi}{\to} & X \\
	\downarrow & & \downarrow f' & & \downarrow f \\
	y \in  E_{Y} & \subset & Y'  & \stackrel{\phi}{\to} & Y 
	\end{array}
	$$
	Take an $\R$-rational function $\varphi$ so that $D'=D +  (\varphi)$ does not contain any image of $\sum_k E_{k}$ in its support. This is possible because there are only finitely many such images and $X$ is affine. In particular, $f_* D'$ does not contain the image of $E_Y$ in its support by the definition of proper push-forwards. So, near the point $y$ we have $f'_* \pi^* D'=0= \phi^* f_* D'$. It remains to verify $f'_* \pi^* (\varphi)=\phi^* f_* (\varphi)$ which follows from the fact that the proper push-forward of a principal divisor of a non-zero rational function $(\varphi)$ is that of its norm $N_{K(X)/K(Y)}(\varphi)$ (see \cite[Proposition 1.4]{fulton}).
\end{proof}

\begin{lem}\label{lem-nef-abund-3}
    Let $f:X \to Z$ be a proper morphism from a normal variety to a variety, $g:X \to Y$ be a surjective finite morphism of normal varieties over $Z$, and $D$ be an $\R$-Cartier divisor on $X$. Suppose $g_*D$ is $\R$-Cartier.  
    \begin{enumerate}
    	\item If $D$ is semi-ample over $Z$, then so is $g_*D$. 
    	
    	\item If $D$ is nef and abundant over $Z$, then so is $g_*D$. 
    \end{enumerate}
\end{lem}
\begin{proof}
	(1) is obvious. (2) follows from Lemmas \ref{lem-nef-abundant-divisor}(3) and \ref{lem-descend}.
\end{proof}

\begin{lem}[From global to local]\label{lem-global-local-nef-abun}
	Let $f:X \to Z$ be a surjective morphism from a normal projective variety to a variety and $D$ be a nef and abundant $\R$-Cartier divisor. Then, $D$ is nef and abundant over $Z$.
\end{lem}
\begin{proof}
	By Lemma \ref{lem-nef-abundant-divisor}, there exist a smooth birational model $\pi: X' \to X$ and an effective divisor $E$ such that $\pi^*D=A_m +\frac{1}{m}E$ with $A_m$ semi-ample for every $m \in \N$. By Lemma \ref{lem-semi-ample-divisor}, $A_m$ is semi-ample over $Z$. Let $Y_m,Y_{mk}$ be the normal varieties over $Z$ given by $A_m,A_{mk}$ respectively, for some integer $k \ge 2$, and $Y$ be a common resolution of them. Replacing $X'$ we may assume $g:X' \dashrightarrow Y$ is a morphism. It follows that $E\sim_\R g^*E_Y$ for some effective divisor $E_Y$ on $Y$, hence $\pi^*D =g^*B$ for some nef and big$/Z$ divisor which completes the proof by Lemma \ref{lem-nef-abundant-divisor} again.
\end{proof}

\begin{lem}
	Let $f\colon X\to Z$ be a projective morphism from a normal variety to a variety, and let $D_1,D_2$ be two $\mathbb{R}$-Cartier divisors on $X$. If both $D_1,D_2$ are nef and abundant, then so is $D_1+D_2$.
\end{lem}
\begin{proof}
	This is immediate from Lemma \ref{lem-nef-abundant-divisor}(3).
\end{proof}

\begin{lem}\label{lem-nef-abund-2}
		Let $f\colon X\to Z$ be a projective morphism from a normal variety to a variety, and let $D$ be an $\mathbb{R}$-Cartier divisor on $X$. Suppose $D$ is a non-negative $\R$-linear combination of finitely many $\Q$-Cartier divisors which are nef and abundant over $Z$. Then, $D$ is nef and abundant over $Z$. Moreover, the converse holds if there exists a boundary divisor $B$ such that $(X,B)$ is lc and $D \sim_\R K_X+B$.
\end{lem}
\begin{proof}
	The first statement follows immediately from the previous lemma. We prove the second statement. Let $\pi:(X',B') \to X$ be a $\Q$-factorial dlt model. Since $D$ is nef and abundant over $Z$, by Lemma \ref{lem-nef-abundant-divisor}, there exist a birational model $\pi': X'' \to X'$, a surjective morphism $g : X'' \to Y$ of smooth quasi-projective varieties over $Z$,
	and a nef and big$/Z$ divisor $G$ of $Y$ such that $(\pi \circ \pi')^*D \sim_\R g^*G$.
	
	Now let $\{D_i\}_{i=1}^n$ be the set of all components of $\pi^*D$ and $\{G_i\}_{i=1}^m$ be the set of all components of $G$. Write $$D+ \sum_{j=1}^k a_j(\varphi_j)=K_{X}+B=\sum_{j=1}^{q} \ell_jL_j;~ (\pi \circ \pi')^*D + \sum_{j=1}^l b_j(\phi_j)=\sum_{j=1}^{p}c_j g^*G_j$$ 
	Consider the following $\R$-linear space
	$$
	\mathcal{L}:=\{ \sum_{j=1}^{n} \alpha_j D_j|\sum_{j=1}^{n} \alpha_j D_j +\sum_{j=1}^k \R(\varphi_j)\text{ intersects } \sum_{j=1}^q \R \pi^*L_j \}
	$$
	which is a rational $\R$-linear subspace of $\sum_{j=1}^n \R D_j$. Moreover, by \cite[Proposition 3.2]{birkar-existII}, there is a rational polytope $\mathcal{N}$ containing $B'$ such that $K_X+\Delta$ is nef$/Z$ for every $\Delta \in \mathcal{N}$. Since $G$ is big, there exists a rational polytope $\mathcal{P}$ containing $G$. Given an element $D'$ of $\mathcal{L}$, because the conditions that $\pi'^* D'+ \sum_{j=1}^l \R(\phi_j)$ intersects $g^*\mathcal{P}$ and that $D'+\sum_{j=1}^k \R(\varphi_j)$ intersects $\sum_{j=1}^q \R \pi^*L_j \bigcap (K_X+ \mathcal{N})$ are rationally polyhedral, we deduce $\pi^*D$ can be expressed as a convex combination of nef and abundant$/Z$ $\Q$-divisors which in turn implies the lemma by construction.
\end{proof}

\begin{rem}
	We note that the converse statement of the previous lemma does not hold unconditionally by an easy example. With a little more effort, the above lemma can be generalised to NQC g-pairs. 
\end{rem}

\begin{defn}[b-nef and log abundant divisors]\label{defn-b-nef-log-abundant-divisor}
	Let $f\colon X\to Z$ be a proper morphism from a normal variety to a variety, $(X,B+M)$ be a g-sub-lc sub-pair with data $\mathbf{M}$, and $\mathbf{D}$ be an $\mathbb{R}$-b-Cartier b-divisor on $X$. We say $\mathbf{D}$ is b-nef and log abundant if for any data log resolution $(\overline{X},\overline{B}+\overline{M}) \to X$ to which $\mathbf{D},\mathbf{M}$ descends, we have $\mathbf{D}_{\overline{X}}$ is nef and log abundant. Note that the definition is independent of the choice of $(\overline{X},\overline{B}+\overline{M})$.
\end{defn}

\subsection{Minimal model program for relatively trivial pairs}\label{subsec-LMMP-scaling}
In this subsection we review the minimal model program (MMP) on a generalised log canonical divisor from \cite{birkarzhang}. We will use standard results of MMP (cf. \cite{kollar-mori}\cite{bchm}).

\noindent \textbf{MMP with scaling.}
Let $(X/Z,B+C+M)=(X_1/Z, B_1 + C_1+M_1)$ be a g-lc pair with data $\overline{M}$ such
that $K_{X_1} + B_1 + C_1+M_1$ is nef$/Z$, $B_1 \ge 0$, and $C_1 \ge 0$ is $\R$-Cartier. Suppose $X \to Z$ is projective, $X$ is $\Q$-factorial and $C$ is big$/Z$. Assume 
\begin{enumerate}
	\item[($*$)] for any $s \in (0, 1)$ there is a boundary $\Delta\sim_\R B + sC + M/Z$ such that
	$(X, \Delta + (1 -s)C)$ is klt.
\end{enumerate}
(Note that $(*)$ is satisfied if $X$ is klt and $C$ is ample$/Z$, which is the only situation in this paper.) 
By minimal model theory, either $K_{X_1} + B_1+M_1$ is nef$/Z$ or there is an extremal ray $R_1/Z$ such that
$(K_{X_1} + B_1+M_1) \cdot R_1 < 0$ and $(K_{X_1} + B_1 + \lambda_1 C_1) \cdot R_1 = 0$ where
$$\lambda_1 := \inf \{t \ge 0 | K_{X_1} + B_1 + t C_1 +M_1 \mathrm{~is~ nef}/Z \}.$$
Now, if $K_{X_1} + B_1+M_1$ is nef$/Z$ or if $R_1$ defines a Mori fibre structure, we stop.
Otherwise, $R_1$ gives a divisorial contraction or a log flip $X_1 \dashrightarrow X_2$.
We now consider $(X_2/Z, B_2 + \lambda_1 C_2+M_2)$ where $B_2 + \lambda_1 C_2+M_2$ is the birational
transform of $B_1 + 
\lambda_1 C_1+M_1$ and continue. By continuing this process, we obtain a sequence of numbers $\lambda_i$ and an MMP$/Z$ which is called an MMP$/Z$ on $K_{X_1} + B_1+M_1$ with scaling of $C_1$.
Note that by definition $\lambda_i \ge \lambda_{i+1}$ for every $i$, and we usually put $\lambda = \lim\limits_{i \rightarrow \infty} \lambda_i$.

Recall that, given a proper morphism $g:X \to Z$, a Cartier divisor $D$ is called \emph{movable over $Z$} if $g_*\mathcal{O}_X(D) \neq 0$ and if the cokernel of the natural homomorphism $g^*g_*\mathcal{O}_X(D) \otimes\mathcal{O}_X(-D) \to \mathcal{O}_X$, called the \emph{base ideal} of $|D/Z|$, has a support of codimension $\ge2$ (\cite[Definition 2.4.3]{fujino-book}). In the geometric context, when $Z$ is quasi-projective (for instance, $Z$ is affine), by a theorem of Bertini (\cite[Corollary 10.9]{hartshorne}), $D$ is movable$/Z$ if and only if the relative base locus $\mathrm{Bs}(|D/Z|)$ has codimension $\ge2$. An $\R$-Cartier divisor $D$ on $X$ is called \emph{movable over $Z$} if it is a non-negative $\R$-linear combination of Cartier movable$/Z$ divisors.

In general, when $K_X+B+M$ is pseudo-effective, we do not know whether the above MMP terminates, but 
we know that in some step of the MMP we reach a model $Y$ on which $K_Y+B_Y+M_Y$, 
the push-down of $K_X+B+M$, is a \emph{pseudo-movable$/Z$} divisor (i.e.a limit of movable$/Z$ divisors in $N^1(X)$): indeed, if the MMP terminates, then 
the claim is obvious; otherwise the MMP produces an infinite sequence $X_i\bir X_{i+1}$ 
of flips and a decreasing sequence $\lambda_i$ of numbers in $(0,1]$ such that 
$K_{X_i}+B_i+\lambda_iC_i +M_i$ is nef$/Z$; by \cite[Lemma 4.4(2)]{birkarzhang}(cf.\cite{bchm}\cite[Theorem 1.9]{birkar-flip}), $\lim\lambda_i=0$; 
in particular, if $Y:=X_1$, then $K_Y+B_Y+M_Y$ is the limit of the movable$/Z$ $\R$-divisors 
$K_Y+B_Y+\lambda_i C_Y+M_Y$.

\noindent \textbf{MMP on very exceptional divisors.}\label{subsec-very-exc}
With the discussion above on reaching a pseudo-movable model, one can prove the following lemmas of MMP on very exceptional divisors. 

\begin{lem}[\text{cf.\cite[Theorem 3.4]{birkar-flip}}]\label{lem-exc-2}
	Let $(X/Z, B+M)$ be a g-lc pair with data $\overline{M}$ such
	that $K_X + B+M \equiv E/Z$ with $E \ge 0$ very exceptional$/Z$. Suppose $X \to Z$ is projective and $X$ is $\Q$-factorial klt. Then, any MMP$/Z$
	on $K_X + B+M$ with scaling of an ample$/Z$ divisor terminates with a good minimal model $Y$ on
	which $K_Y + B_Y +M_Y \equiv E_Y = 0/Z$.
\end{lem}
\begin{proof}
	Since the question is local, by shrinking $Z$ we may assume $Z$ is affine. Run an MMP$/Z$ on $K_X+B+M$ with scaling of an ample$/Z$ divisor $C$. The only
	divisors that can be contracted are the components of $E$ hence $E$ remains very
	exceptional$/Z$ during the MMP. After finite steps, we reach a model $Y$ on which $K_Y+B_Y+M_Y \equiv E_Y$, the push-down of $K_X+B+M$ and $E$, is pseudo-movable$/Z$. Since $K_Y+B_Y+M_Y$ is nef on the very general curves of
	$S/Z$ for any component of $E_Y$, and $E_Y$ is very exceptional$/Z$, by the negativity lemma \cite[Lemma 3.3]{birkar-flip}, we deduce $E_Y=0$.
\end{proof}

With a little more work we can obtain a slightly generalised result. Notation as below, we note that, by \cite[V. 1.12. Corollary]{nakayama}, for a very general fibre $F$, the condition $E^h|_{F} = N_\sigma((K_X+B+M)|_{F})$ is equivalent to that $\kappa_{\sigma}((K_X+B+M)|_{F})=0$. So it does not depend on the choice of $F$, and is preserved under MMP.

\begin{lem}\label{lem-exc-2+}
	Let $(X/Z, B+M)$ be a g-lc pair with data $\overline{M}$ such
	that $K_X + B+M \equiv E/Z$ with $E \ge 0$ very exceptional$/Z$. Writing $E=E^h+E^v$ where $E^h$ denotes the horizontal$/Z$ part and $E^v$ denotes the vertical$/Z$ part, we suppose that 
	\begin{itemize}
		\item $X \to Z$ is projective and $X$ is $\Q$-factorial klt,
		
		\item $E^h|_{F} = N_\sigma((K_X+B+M)|_{F})$, where $F$ is a very general fibre of the Stein factorisation of $X \to Z$, and
		
		\item $E^v$ is very exceptional$/Z$.
	\end{itemize}
	Then, any MMP$/Z$
	on $K_X + B+M$ with scaling of an ample$/Z$ divisor terminates with a good minimal model $Y$ on
	which $K_Y + B_Y +M_Y \equiv E_Y = 0/Z$.
\end{lem}
\begin{proof}
	With a similar argument as above, we run an MMP$/Z$ and reach a model $Y$ on which $K_Y+B_Y+M_Y \equiv E_Y$, the push-down of $K_X+B+M$ and $E$, is pseudo-movable$/Z$. It follows that, for a very general fibre, we have $E_Y|_F=E_Y^h|_F$ is pseudo-movable, and thus one infers $E_Y^h=0$ which in turn implies that $E_Y$ is very exceptional$/Z$. So the lemma is proved by Lemma \ref{lem-exc-2}.
\end{proof}

The same argument also imply:
\begin{lem}[\text{\cite[Proposition 3.8]{hanli}, cf.\cite[Theorem 3.5]{birkar-flip}}]\label{lem-exc-3}
	Let $(X/Z, B+M)$ be a g-lc pair with data $\overline{M}$ such that $X \rightarrow Z$ is
	projective, $X$ is $\Q$-factorial klt, and $K_X + B+M \equiv Q = Q_+ - Q_-/Z$ where $Q_+$, $Q_- \ge 0$ have
	no common components and $Q_+$ is very exceptional$/Z$. Then, any MMP$/Z$ on
	$K_X + B+M$ with scaling of an ample$/Z$ divisor contracts $Q_+$ after finite steps.
\end{lem}

\begin{lem}[\text{G-dlt modification, \cite[Lemma 4.5]{birkarzhang}\cite[Proposition 3.9]{hanli}}]\label{lem-g-dlt-blow-up}
	Let $(X/Z, B+M)$ be a g-lc pair with data $\overline{M}$. Then, there exists a birational model $\pi:X' \to X$ such that $(X'/Z, B'+M')$ be a $\Q$-factorial g-dlt pair with data $\overline{M}$, where $K_{X'}+B'+M' =\pi^*(K_X+B+M)$, and that $a(E,X,B+M)=0$ for every exceptional prime divisor $E$.
\end{lem}
\begin{proof}
	This is almost a direct consequence of Lemma \ref{lem-exc-3}. It remains to check that $(X_i,B_i+M_i)$ being g-dlt is preserved under MMP, which follows directly from the fact that, the locus contracted by the MMP does not contain any g-lc centre. 
\end{proof}

\section{Lc-trivial morphisms and moduli b-divisors}\label{sec5}

In this section we generalise the notion of lc-trivial fibrations to proper surjective morphisms, and then we define the discriminant b-divisor and moduli b-divisor. Under these constructions we study the positivity of the moduli b-divisor of a class of lc-trivial morphisms, namely good lc-trivial morphisms. 

\subsection{Lc-trivial fibrations}\label{subsec-lc-trivial-fib}
In this subsection we review some definitions and results from \cite{ambro0}\cite{ambro1} \cite{ambro2} with a generalisation to $\R$-divisors. 

\noindent \textbf{Pre-discriminant divisors.}
Let $f : X \to Y$ be a proper surjective morphism of normal varieties and $(X ,B)$ is sub-lc over the generic point $\eta_Y$ of $Y$. For a prime divisor $P \subset Y$. By shrinking $Y$ around the generic point of $P$, we assume
that $P$ is Cartier. We set
\begin{align*}
b_P = \max \{ t\in \R| \text{$(X, B + tf^*P)$ is sub-lc over
	the generic point of $P$} \}
\end{align*}
and set
$$
B_Y =\sum_P (1 - b_P )P,
$$
where $P$ runs over prime divisors on $Y$. Then it is easy to see that $B_Y$ is well defined since $b_P = 1$ for all but a finite number of prime divisors and it is called the \emph{pre-discriminant divisor}. 

If $K_X+ B \sim_{\mathbb{R}} f^*D$ for some divisor $D$ on $Y$, then we set
$$
M_Y = D- K_Y-B_Y
$$
and call $M_Y$ the \emph{pre-moduli divisor}. Note that the pre-discriminant divisor (resp. pre-moduli divisor) is called the discriminant divisor (resp. moduli divisor) in literatures such as \cite{ambro0}\cite{ambro1} etc.. 
For basic properties of the pre-discriminant divisors, we refer to \cite[Remark 3.1, Example 3.1]{ambro0}.

Let $\phi : Y' \to Y$ be a birational contraction from a normal variety $Y'$. Let $X'$ be a resolution of the main component of $X \times_Y Y'$ which
dominates $Y'$. The induced morphism $\pi: X' \to X$ is birational, and $K_{X'}+B'=\pi^*(K_X+B)$. Let $B_{Y'}$ be the pre-discriminant of $K_{X'} + B'$ on $Y'$. Since
the definition of the pre-discriminant is divisorial and $\phi$ is an isomorphism
over codimension one points of $Y$ , by \cite[Remark 3.1]{ambro0} we have $B_Y = \phi_*(B_{Y'})$.
This means
that there exists a unique b-divisor $\mathbf{B}$ of $Y$ such that $\mathbf{B}_{Y'}$ is the
pre-discriminant on $Y'$ of the induced fibre space $f': (X', B') \to Y'$, for
every birational model $Y'$ of $Y$ . We call $\mathbf{B}$ the \emph{pre-discriminant b-divisor}. We define the \emph{pre-moduli b-divisor} $\mathbf{M}$ in a similar way. 

\noindent \textbf{Lc-trivial fibrations.}
Recall that the discrepancy b-divisor $\mathbf{A} = \mathbf{A}(X, B)$
of a pair $(X, B)$ is the b-divisor of $X$ with the trace $ \mathbf{A}_{X'}$ defined by
the formula
$$
K_{X'}= \pi^*(K_X + B) +  \mathbf{A}_{X'} ,
$$
where $\pi : X' \to X$ is a proper birational morphism of normal varieties.
Similarly, we define $ \mathbf{A}^{\ast} =  \mathbf{A}^{\ast}(X, B)$ by
$$
\mathbf{A}_{X'}^{\ast} = \sum_{a_i>-1} a_i A_i
$$
for
$$
K_{X'} = \pi^*(K_X + B) +\sum a_i A_i.
$$
Note that $\mathbf{A}(X, B) = \mathbf{A}^{\ast}(X, B)$ when $(X, B)$ is sub-klt.

\begin{defn}[\text{cf.\cite[Definition 3.2]{fujino-gongyo2}\cite[Definition 2.1]{ambro1}}]\label{defn-Q-lc-trivial-fib}
	A \emph{$\K$-lc-trivial (resp. $\K$-klt-trivial) fibration} $f : (X, B) \to
	Y$ consists of a proper surjective morphism between normal varieties with connected fibers and a sub-pair $(X, B)$ satisfying the	following properties:
	\begin{enumerate}
		\item $(X, B)$ is sub-lc (resp. sub-klt) over the generic point of $Y$;
		
		\item $\mathrm{rank} f_*\mathcal{O}_X(\lceil  \mathbf{A}^*
		(X, B)\rceil) = 1$;
		
		\item There exists an $\R$-Cartier divisor $D$ on $Y$ such that
		$$
		K_X + B \sim_\K f^*D.
		$$
	\end{enumerate}
	In this case, the pre-discriminant (resp. pre-moduli) divisor (resp. b-divisor) is said to be the \emph{discriminant} (resp. \emph{moduli}) \emph{divisor} (resp. \emph{b-divisor}).
\end{defn}

We briefly sketch the results for $\Q$-lc-trivial fibrations. Note that we allow $D$ to be $\R$-Cartier even for the definition of a $\Q$-lc-trivial fibration. In fact, by modifying $D$ to a $\Q$-Cartier divisor, one can easily reduce to the rational case. Thanks to the important results \cite[Theorem 2.5]{ambro1}\cite[Theorem 3.6]{fujino-gongyo2} obtained by the theory of variations of (mixed) Hodge structure, the moduli b-divisor $\mathbf{M}$ of a $\Q$-lc-trivial fibration is $\Q$-b-Cartier and b-nef. Hence $\mathbf{K}+\mathbf{B}$ is $\R$-b-Cartier.

If we assume further that $Y$ is complete, the geometric generic fiber $X_{\overline{\eta}} = X \times_Y \Spec(\overline{k(Y )})$ is a projective variety and $(X_{\overline{\eta}},B_{\overline{\eta}})$ is klt, where $B_{\overline{\eta}}=B|_{X_{\overline{\eta}}}$, then by \cite[Theorem 3.3]{ambro1}\cite[Theorem 3.10]{fujino-gongyo2} the moduli b-divisor $\mathbf{M}$ is b-nef and abundant.

Moreover, by \cite[Theorem 1.1]{fujino-gongyo2}, the moduli b-divisor $\mathbf{M}$ of a $\Q$-lc-trivial fibration from an lc pair is b-nef and abundant. As we mentioned in the introduction, we will not directly apply \cite[Theorem 1.1]{fujino-gongyo2} in this article, but we will use their strategy in Section \ref{subsec-adj-commute}. 

\begin{lem}\label{l-pull-back-Cartier}
	Notation as in Definition \ref{defn-Q-lc-trivial-fib}, we write
	$$
	K_X + B +(\varphi) = f^*(K_Y+B_Y+M_Y)
	$$
	for a $\K$-rational function $\varphi$. 
	\begin{enumerate}
		\item The moduli b-divisor $\mathbf{M}$ is uniquely determined by $\varphi$.
		
		\item For any $\K$-rational function $\psi$, $K_X + B +(\psi) =0/Y$ if and only if $K_F+B_F +(\psi|_F)=0$, where $F$ is a general fibre (or the geometric generic fibre) and $K_F+B_F=(K_X+B)|_F$.
	\end{enumerate}
   In particular, if $f$ is $\Q$-lc-trivial, then there exists an integer $r$, denoted by $r=b(F,B_F)$, uniquely defined as the least positive integer such that $r(K_F+B_F) \sim 0$. 
\end{lem}
\begin{proof}
	(1) is obvious. (2) follows from the fact that every $\K$-rational function $\varphi$ with vertical support can be decomposed to $h \circ f$ for a $\K$-rational function $h$ on $Y$.
\end{proof}

\begin{lem}\label{lem-Q-lc-trivial-fib}
	Let $f:(X,B) \to Y$ be an $\R$-lc-trivial (resp. $\R$-klt-trivial) fibration. Then, $B$ is a convex combination of $\Q$-divisors $B_i$ such that $f:(X,B_i) \to Y$ is $\Q$-lc-trivial (resp. $\Q$-klt-trivial).
\end{lem}
\begin{proof}
	We only prove for the $\R$-lc-trivial fibrations since the klt case can be argued verbatim. Replacing $X$ we may assume it is smooth. Let $f:(X,B) \to Y$ be an $\R$-lc-trivial fibration, $\varphi=\prod_{i=1}^k \varphi_i^{\alpha_i}$ be an $\R$-rational function so that $K_X+B+(\varphi)=f^*D$. Let $\mathcal{V} \subset \CDiv_\R(Y) $ be a finite dimensional rational linear subspace containing $D$, $\mathcal{L}\subset \CDiv_\R(X)$ be a rational polytope containing $B$ such that, for every $\Delta \in \mathcal{L}$, we have $(X,\Delta)$ is a sub-pair which is sub-lc over the generic point of $Y$. Now we consider the rational polytope 
	$$
	\mathcal{P}:=\{\Delta \in \mathcal{L}|\Delta +\sum_{i=1}^k \R(\varphi_i)\text{ intersects }f^*\mathcal{V} \}
	$$
	For every $\Delta \in \mathcal{P}$, we have further $K_X+\Delta \sim_\R 0/Y$. It is obvious that $B \in \mathcal{P}$.
	
	It suffices to show that, there exists a convex combination $B=\sum_j \alpha_jB_j$ of $\Q$-divisors $B_j \in \mathcal{P}$ with $\mathrm{rank} f_*\mathcal{O}_X(\lceil  \mathbf{A}^*
	(X, B_j)\rceil) = 1$. To this end, pick a log resolution $\pi: \overline{X} \to X$ of $(X, \sum_j \Gamma_j  )$ where every element of $\mathcal{P}$ is supported by $\sum_j \Gamma_j$. Note that the proofs of \cite[Lemmas 3.19 and 3.20]{fujino-abund-saturation} are still valid for $\R$-sub-boundaries. Hence, by shrinking $Y$, we may assume $(X,\Delta)$ is sub-lc for every $\Delta \in \mathcal{P}$, and we have $$
	f_*\mathcal{O}_X(\lceil  \mathbf{A}^*
	(X, \Delta)\rceil)=f_* \pi_* \mathcal{O}_{\overline{X}}(\sum_{a_i \neq -1} \lceil a_i  \rceil A_i)$$
	where $K_{\overline{X}}=\pi^*(K_X+\Delta) +\sum a_i A_i$. 
	Consider the rational sub-polytope 
	$$\mathcal{Q}=\{\Delta \in \mathcal{P}| \lceil  \mathbf{A}^*
	(X, \Delta)_{\overline{X}}\rceil \le \lceil  \mathbf{A}^*
	(X, B)_{\overline{X}}\rceil \}.$$ Then, for any $B_j \in \mathcal{Q}$, we have $\mathrm{rank} f_*\mathcal{O}_X(\lceil  \mathbf{A}^*
	(X, B_j)\rceil) = 1$ which completes the proof. 
\end{proof}

\begin{rem}
	The converse direction of the previous lemma is not true in general. See the example below.
\end{rem}

\begin{exa}
	Let $\pi: X \to \PP^2$ be the blow-up at two points $p_1,p_2\in \PP^2$, $E$ be one of the exceptional curves. Let $L_0$ be the line passing through both two points, $L_{i,j}$ be general lines passing through $p_i$ respectively, for $i=1,2,j=1,2$, and let $L_0',L_{i,j}'$ be the birational transforms on $X$. Let $A$ be a general member of $|L|_\Q$, where $L$ is a general line, and $A'$ be the birational transform of $A$ on $X$.
	
	Set $B_1=3A'-E$ and $B_2=\frac{3}{2}A'+ \frac{3}{4}(L_{1,1}'+L_{1,2}') +\frac{3}{4}( L_{2,1}'+L_{2,2}') -\frac{1}{2}L_0'$. One can easily check that $K_X+B_1 \sim_\Q 0, K_X+B_2 \sim_\Q 0$ with $h^0(X,\mathcal{O}_X(\lceil-B_1\rceil)) =0, h^0(X,\mathcal{O}_X(\lceil-B_2\rceil)) =0$, but $h^0(X,\mathcal{O}_X(\lceil-\frac{1}{2}(B_1+B_2)\rceil))=h^0(X,\mathcal{O}_X(E+L_0'))=1$.   
\end{exa}

The next proposition is a consequence of toroidal geometry.
\begin{prop}\label{prop-R-lc-trivial}
	Notation as in Lemma \ref{lem-Q-lc-trivial-fib}, write $\mathbf{B},\mathbf{B}_j$ (resp. $\mathbf{M},\mathbf{M}_j$) for the discriminant (resp. moduli) b-divisor of $(X,B) \to Y,(X,B_j) \to Y$ respectively. If we write $B= \sum_j \alpha_j B_j$ as the convex combination, then we may further require $\mathbf{B}= \sum_j \alpha_j \mathbf{B}_j$ (resp. $\mathbf{M}= \sum_j \alpha_j \mathbf{M}_j$). 
\end{prop}

Before we prove Proposition \ref{prop-R-lc-trivial}, we need some preparation.
\begin{defn}[Resolved model]
	Let $f:(X,B) \to Y$ be an $\R$-lc-trivial fibration. We say it is \emph{resolved} if the followings hold:
	 	\begin{enumerate}
	 	\item There exist a convex combination $B=\sum_i \alpha_iB_i$ of $\Q$-divisors $B_i$ such that $f:(X,B_i) \to Y$ is $\Q$-lc-trivial with $\Q$-rational functions $\varphi_i^{1/r_i}$ where $r_i=b(F,B_{i,F})$. 
	 		
	 	\item There exist reduced divisors $\Delta,\Delta_Y$ on $X,Y$ respectively, such that $(X,\Delta),(Y,\Delta_Y)$ are quasi-projective log smooth, and $\Delta^v=f^{-1}\Delta_Y$.
	 	
	 	\item $(f^{-1}U,\Delta|_{f^{-1}U})$ is log smooth over $U$, where  $U=Y \setminus \Delta_Y$.
	 	
	 	\item  $B_i$'s and $(\varphi_i)$'s are supported by $\Delta$, and $B_{Y,i}$'s and $M_{Y,i}$'s are supported by $\Delta_{Y}$.
	 \end{enumerate}
\end{defn}

Note that the existence of a resolved model of a given $\R$-lc-trivial fibration is guaranteed by log resolution and Lemma \ref{lem-Q-lc-trivial-fib}. 

Recall that a morphism $f:(X,\Delta) \to Y$ from a log smooth pair with a reduced divisor $\Delta$ to a normal variety is \emph{a log smooth morphism} if for any stratum $S$, the restriction map $f|_S$ is a smooth morphism. In this case, we say $(X,\Delta)$ \emph{log smooth over} $Y$.

\begin{lem}\label{lem-resolved-model}
	Let $(X,\Delta)$ be a log smooth pair with a reduced divisor $\Delta$ and $f:(X,\Delta) \to Y$ be a dominant morphism to a smooth variety. Given a rational function $\varphi \in k(X)$ with $(\varphi)$ supported in $\Delta$ and a positive integer $l$, let $X'$ be the normalisation of $X$ in $k(X)(\varphi^{\frac{1}{l}})$ and $\Delta'$ be the inverse image of $\Delta$. Then, there exists a log resolution $\pi:X'' \to X'$ of $(X',\Delta')$ such that 
	\begin{enumerate}
		\item The composite morphism $f'':(X'',\Delta'') \to Y$ is log smooth.
		
		\item If we denote by $\Delta''=\pi^{-1}\Delta'$, then $K_{X''}+\Delta''$ is the pull-back of $K_X+\Delta$.
	\end{enumerate}
\end{lem}

\begin{proof}
	By Lemma \ref{lem-toroidal}, we have $f:(X,\Delta) \to (Y ,0)$ is toroidal (see Notation \ref{note-toroidal}). Since $\mathrm{div}\varphi$ is supported in $\Delta$, we deduce that the normalisation $\gamma:(X',\Delta') \to (X,\Delta)$ is toroidal by Abhyankar's lemma (\cite[Lemma 3.3]{adk}). Taking a toroidal resolution $X''$ of $X'$, one immediately concludes (1). The assertion (2) follows from the ramification formula.
\end{proof}

\begin{proof}[Proof of Proposition \ref{prop-R-lc-trivial}]
	Replacing $X$, $Y$ and $f$, we may assume it is resolved. Let $\mathcal{P} \subset \CDiv_\R(X)$ be the polytope defined by $B_i$'s. For any prime divisor $P$ on $Y$, we set the function $b_P$ on $\mathcal{P}$: 
	\begin{align*}
	b_P(C) = \max \{ t\in \R| \text{$(X, C + tf^*P)$ is sub-lc over
		the generic point of $P$} \}.
	\end{align*}
	We note that the $b_P$ is a convex piecewisely affine function and gives a rational polyhedral decomposition of $\mathcal{P}$. Also note that $b_P$ is not identically one on $\mathcal{P}$ when $P \nsubseteq \Delta_Y$. Therefore, there exists a rational sub-polytope $\mathcal{Q}$ containing $B$ such that $b_P$ is affine on $\mathcal{Q}$, for any prime divisor $P$. In particular, replacing $B_i$'s and $\alpha_i$'s, we have $B_Y=\sum_i \alpha_i B_{Y,i}$ and $M_Y=\sum_i \alpha_i M_{Y,i}$, where $B_Y,B_{Y,i}$ are discriminant divisors and $M_Y,M_{Y,i}$ are moduli divisors of $f:(X,B) \to Y,f:(X,B_i)\to Y$ respectively.
	
	By Lemma \ref{lem-resolved-model} and \cite[Section 5]{ambro1}, all $\mathbf{M}_i$'s descend to $Y$ which concludes the proposition.
\end{proof}

\begin{rem}
	The argument above shows that the base variety $Y$ of a resolved model is an \emph{Ambro model}, that is, $\mathbf{M}$ descends to $Y$.
\end{rem}

By Lemma \ref{lem-Q-lc-trivial-fib} and Proposition \ref{prop-R-lc-trivial}, an $\R$-lc-trivial fibration is a convex combination of $\Q$-lc-trivial fibrations, and in the same way, its discriminant (resp. moduli) b-divisor is a convex combination of $\Q$-b-divisors. We immediately obtain the following theorem.
\begin{thm}[cf.\text{\cite[Theorem 2.5]{ambro1}}]\label{thm--lc-trivial}
	With notation in Definition \ref{defn-Q-lc-trivial-fib}, letting $\mathbf{B}$ be the discriminant b-divisor, we have $\mathbf{B}=\sum_i \alpha_i \mathbf{B}_{Y,i}$ is a convex combination of $\Q$-b-divisors $\mathbf{B}_i$ such that $\mathbf{K}+\mathbf{B}_i$ is $\Q$-b-Cartier. In particular, $\mathbf{K}+\mathbf{B}$ is $\R$-b-Cartier. In the same way, the moduli b-divisor $\mathbf{M}=\sum_i \alpha_i \mathbf{M}_{Y,i}$ is a convex combination of $\Q$-b-Cartier and b-nef $\Q$-divisors. In particular, $(Y,B_Y+M_Y)$ with data $\mathbf{M}$ is an NQC g-sub-pair.
\end{thm}

One can immediately generalise the results from \cite{ambro1}\cite{fujino-gongyo2} as below.
\begin{cor}[\text{cf.\cite[Theorem 3.3]{ambro1}\cite[Theorem 3.10]{fujino-gongyo2}}]\label{cor--klt-trivial}
	Let $f:(X,B) \to Y$ be an $\R$-lc-trivial fibration. Suppose that $Y$ is complete, the geometric generic fiber $X_{\overline{\eta}} = X \times_Y \Spec(\overline{k(Y )})$ is a projective variety and $(X_{\overline{\eta}},B_{\overline{\eta}})$ is klt, where $B_{\overline{\eta}}=B|_{X_{\overline{\eta}}}$. Then, the moduli b-divisor $\mathbf{M}$ is a convex combination of $\Q$-b-Cartier and b-nef and abundant $\Q$-divisors. In particular, $\mathbf{M}$ is $\R$-b-Cartier and b-nef and abundant.
\end{cor}
\begin{proof}
	This is immediate by Proposition \ref{prop-R-lc-trivial}. The last assertion follows from Lemma \ref{lem-nef-abund-2}.
\end{proof}

As a consequence of Theorem \ref{thm--lc-trivial}, one can easily generalise a canonical bundle formula of Fujino-Mori type (cf. \cite{fujino-mori}) and \cite[Corollary 1.1.2]{bchm}. For details of the proof, we refer to \cite{hu2}.

\begin{thm}[\text{\cite[Theorem 1.2]{hu2}}]\label{thm-canonical-bundle-formula}
	Let $f:X \to Y$ be a contraction of normal varieties and $(X,B)$ be an lc pair such that $\kappa_{\iota}(X/Y,K_X+B)=0$. Then, there exists a commutative diagram
	$$
	\xymatrix{
		(X',B') \ar[d]_{f'} \ar[r]^{\pi}   &  (X,B)\ar[d]^{f}\  &\\
		Y' \ar[r]^{\phi} &    Y } 
	$$  
	which consists of birational models $\pi:X' \to X$ and $\phi:Y' \to Y$, such that:
	\begin{enumerate}
		\item $K_{X'}+B'=\pi^*(K_X+B)+E$ where $E$ is exceptional$/X$ and $B',E\ge 0$ have no common components.
		
		\item $K_{X'}+B' \sim_{\mathbb{R}} f'^*(K_{Y'}+B_{Y'}+M_{Y'})+R$ where $R \ge 0$ and $(Y',B_{Y'}+M_{Y'})$ is a g-lc pair with data $\mathbf{M}$.
		
		\item $\kappa(X'/Y',R^h)=0$ and $R^v$ is very exceptional$/Y'$, where $R^h$ (resp. $R^v$) denotes the horizontal (resp. vertical) part over $Y'$.
	\end{enumerate}
	Moreover, if every lc centre of $(X,B)$ is horizontal$/Y$, then $(Y',B_{Y'}+M_{Y'})$ is g-klt.
\end{thm}
\begin{thm}[\text{\cite[Theorem 1.1]{hu2}}]\label{thm-canonical-model}
	Let $(X/Z,B)$ be a klt pair with the Kodaira dimension $\kappa_\iota(X/Z,K_X+B) \ge 0$. Then, $(X/Z,B)$ has the canonical model.
\end{thm}

Because an $\R$-lc-trivial fibration possesses almost the same properties, from now on we will omit $\R$ and simply say it an \emph{lc-trivial fibration}. The category of lc-trivial (resp. klt-trivial) fibrations is closed under base change (see \cite[Proposition 3.1]{ambro2}\cite[Section 3.3]{fujino-gongyo2}). More precisely, given an lc-trivial fibration $f : (X, B) \to Y$ and a surjective proper morphism $\gamma: Y' \to Y$, the \emph{induced lc-trivial fibration} $f':(X',B') \to Y'$
is given by the morphism induced by base change and $K_{X'}+B'=\rho^*(K_X+B)$ where $\rho: X' \to X$ is the induced map.

\begin{lem}[\text{cf.\cite[Proposition 3.1]{ambro2}}]\label{lem-lc-fib-base-change}
	Let $f : (X, B) \to Y$ be an lc-trivial fibration. Let $\gamma: Y' \to Y$ be a surjective proper morphism from a normal variety $Y'$, and let $f': (X', B') \to Y'$
	be an lc-trivial fibration induced by base change
	$$
	\xymatrix{
		(X',B') \ar[d]_{f'} \ar[r]^{\rho}   &  (X,B)\ar[d]^{f}  &\\
		Y' \ar[r]^{\gamma} &  Y } 
	$$  
	Let $\mathbf{M}$ and $\mathbf{M}'$
	be the corresponding moduli b-divisors. Then,	
	$$	
	\gamma^*\mathbf{M} =\mathbf{M}' .
	$$
\end{lem}

\subsection{Lc-trivial morphisms.}\label{subsec-lc-trivial-morphism}
We consider a class of morphisms as follows. 
\begin{defn}\label{defn-lc-trivial-morphism}
	An \emph{lc-trivial (resp. klt-trivial) morphism} $f : (X, B) \to
	Y$ consists of a proper surjective morphism between normal varieties and a sub-pair $(X, B)$ satisfying the
	following properties:
	\begin{enumerate}
		\item $K_X + B \sim_\K 0/Y$,
		
		\item If we denote by $X \overset{\widetilde{f}}{\to} \widetilde{Y} \overset{\gamma}{\to} Y$ the Stein factorisation, then  $\widetilde{f} : (X, B) \to
		\widetilde{Y}$ is an lc-trivial (resp. klt-trivial) fibration (see Definition \ref{defn-Q-lc-trivial-fib}).
	\end{enumerate}
\end{defn}

To study these objects, we want to establish a canonical bundle formula of Kodaira type. Unfortunately, the classical definition of the discriminant divisor usually disables the positivity of the moduli b-divisor (see \cite[Remark 3.2]{ambro0}). We therefore introduce the following definition of the discriminant divisor of a proper surjective generically finite morphism. The construction is from \cite[Theorem 4.5]{hanliu} and \cite[Lemma 1.1]{fg-bundle}.

\begin{defn}[\text{\cite[Theorem 4.5]{hanliu}\cite[Lemma 1.1]{fg-bundle}}]\label{defn-pushdown}
	Let $(\widetilde{X}/Z,\widetilde{B}+\widetilde{M})$ be a g-sub-pair, $\widetilde{\mathbf{B}}$ be its pre-boundary $\R$-b-divisor, $\widetilde{\mathbf{M}}$ be its moduli $\R$-b-Cartier b-divisor and $f : \widetilde{X} \to X$ be a proper surjective generically finite morphism over $Z$ between normal varieties such that
	$$
	K_{\widetilde{X}} + \widetilde{B} + \widetilde{M} \sim_{\mathbb{K}} 0/X.
	$$
	We define the g-sub-pair $(X, B + M)$ by specifying the traces on the
	higher models of $X$ as follows: for any birational morphism $\phi: X' \to X$,
	consider the following commutative diagram
	$$
	\xymatrix{
		\widetilde{X}' \ar[d]_{f'} \ar[r]^{\pi}   &  \widetilde{X}\ar[d]^{f}  &\\
		 X'\ar[r]^{\phi} &   X } 
	$$
	where $f'$ is induced by base change. Define
	$$
	\mathbf{B}_{X'} =\frac{1}{\deg f}f'_*(\widetilde{R} + \widetilde{B}'), \text{~and~} \mathbf{M}_{X'} =\frac{1}{\deg f}f'_*\widetilde{M}',
	$$
	as the proper push-forwards (\cite[1.4]{fulton}) where $\widetilde{B}'=\widetilde{\mathbf{B}}_{\widetilde{X}'}$, $\widetilde{M}'=\widetilde{\mathbf{M}}_{\widetilde{X}'}$ and $\widetilde{R}$ is the closure of ramification divisor of $f'$ restricted over the Gorenstein locus of $X'$. 
\end{defn}

\begin{rem}\label{rem-gen-pullback}
	In the above notation, for every pair of birational models $\widetilde{X}' \to \widetilde{X}$ and $X' \to X$ such that the induced map $f':\widetilde{X}' \dashrightarrow X'$ is a morphism, by \cite[Proof of Lemma 1.1]{fg-bundle} one may easily verify that 
	$$
	K_{\widetilde{X}'} + \widetilde{B}' + \widetilde{M}' \sim_{\mathbb{K}}f'^*( K_{X'}+B' + M' )
	$$
	where $\widetilde{B}',\widetilde{M}'$ are the traces of $\widetilde{\mathbf{B}},\widetilde{\mathbf{M}}$ on $\widetilde{X}'$ and $B',M'$ are the traces of $\mathbf{B},\mathbf{M}$ on $X'$. Indeed, if we write $K_{\widetilde{X}'} + \widetilde{B}' + \widetilde{M}'\sim_{\mathbb{K}} f'^*D'$ for some $\R$-Cartier divisor $D'$, since $f'$ is finite flat over an open subset of $X'$ outside a locus of codimension $\ge 2$, then by flattening lemma and \cite[\href{https://stacks.math.columbia.edu/tag/02RH}{Lemma 02RH}]{stacks-project} (see also \cite[Example 1.7.4]{fulton}) we see 
	$$f'_*(K_{\widetilde{X}} + \widetilde{B} + \widetilde{M}) = f'_*f'^*( K_{X'}+B'+M')=(\deg f) ( K_{X'}+B'+M')$$ which in turn implies that $D' \sim_\K K_{X'}+B'+M'$. So the formula follows. Note that the same argument also works for the relation ``$\sim$" and ``$=$" instead of ``$\sim_\K$".
\end{rem}
	
In order to show that $(X,B+M)$ is actually a g-sub-pair, it remains to verify that the b-divisors $\mathbf{K}+\mathbf{B}$ and $\mathbf{M}$ descends to some model and $\mathbf{M}$ is represented by a nef divisor. 
Quite recently, J. Han and W. Liu also obtained a similar result \cite[Theorem 4.5]{hanliu}.
\begin{thm}\label{thm-gen-pair}
	Let $(\widetilde{X}/Z, \widetilde{B} + \widetilde{M})$ be a g-sub-pair with data $\widetilde{\mathbf{M}}$ over a variety $Z$. Let $f : \widetilde{X} \to	X$ be a proper surjective generically finite morphism of normal varieties over $Z$ with $K_{\widetilde{X}} + \widetilde{B} + \widetilde{M} \sim_{\mathbb{K}} 0/X$. Then, there is a g-sub-pair $(X/Z, B + M)$ with data $\mathbf{M}$ such that
	$$
	K_{\widetilde{X}} + \widetilde{B} + \widetilde{M} \sim_{\mathbb{K}}f^*( K_X + B + M)
	$$
	Moreover, if $(\widetilde{X}/Z, \widetilde{B} + \widetilde{M})$ is g-lc (resp. g-klt, g-sub-lc, g-sub-klt), then so is $(X/Z, B + M)$; and if $\widetilde{\mathbf{M}}$ is b-nef and abundant$/Z$ (resp. semi-ample$/Z$), then so is $\mathbf{M}$.
\end{thm}
\begin{proof}
	By Lemma \ref{lem-descend}, we infer that $\mathbf{M}$ descends to some model $X' \to X$. In fact, by flattening lemma (a special case of Theorem \ref{thm-equidimension}), we can assume the induced morphism $f':\widetilde{X}' \to X'$ is finite, $X'$ is smooth and $\widetilde{\mathbf{M}}$ descends to $\widetilde{X}'$. Hence, we deduce $M'=\mathbf{M}_{X'}$ is nef$/Z$ which completes the first part. 
	
	To show the second part, we simply notice that, by definition $B_{Y'}=\frac{1}{\deg f}(R'+B')$ and apply \cite[Proof of Lemma 1.1]{fg-bundle} to conclude the result on singularities. The last assertion follows from Lemma \ref{lem-nef-abund-3}.
\end{proof}

With the above construction we are ready to define the moduli b-divisors of lc-trivial morphisms.

\begin{defn}\label{defn-moduli}
	Let $f : (X, B) \to	Y$ be an lc-trivial morphism, $X \overset{\widetilde{f}}{\to} \widetilde{Y} \overset{\gamma}{\to} Y$ be the Stein factorisation, $\varphi$ be a $\K$-rational function, $D$ be an $\R$-Cartier divisor on $Y$ and
	$$
	K_X+B+(\varphi)=\widetilde{f}^*(K_{\widetilde{Y}}+B_{\widetilde{Y}}+M_{\widetilde{Y}})=f^*D
	$$
	with the pre-boundary b-divisor $\widetilde{\mathbf{B}}$ and data $\widetilde{\mathbf{M}}$. By Definition \ref{defn-pushdown}, Remark \ref{rem-gen-pullback} and Theorem \ref{thm-gen-pair} we define a pre-boundary b-divisor $\mathbf{B}$ and data $\mathbf{M}$ satisfying 
	$$
	K_X+B+(\varphi)=f^*(K_Y+B_Y+M_Y)
	$$
	and $(Y,B_Y+M_Y)$ is a g-sub-pair where $B_Y,M_Y$ are the traces of $\mathbf{B},\mathbf{M}$ on $Y$. We call $\mathbf{B}$ the \emph{discriminant b-divisor} and $\mathbf{M}$ the \emph{moduli b-divisor}.
\end{defn}

\begin{rem}\label{rem-pushdown}
	Notation as above, we remark the followings:
	\begin{enumerate}
		\item The discriminant b-divisor $\mathbf{B}$ (resp. the moduli b-divisor $\mathbf{M}$) defined above does not in general coincide with that in \cite{ambro0} (the pre-discriminant (resp. pre-moduli b-divisor) in Section \ref{subsec-lc-trivial-fib}). As pointed out in \cite[Remark 3.2]{ambro0} the classical construction does not guarantee the b-nefness of $\mathbf{M}$.
		
		\item If $(X,B)$ is lc (resp. klt, sub-lc, sub-klt), then by Theorem \ref{thm-gen-pair}, $(Y,B_Y+M_Y)$ is a g-lc (resp. g-klt, g-sub-lc, g-sub-klt) pair. 
	\end{enumerate}
\end{rem}

\begin{defn}\label{defn-good-lc-trivial}
	With the notation of Definition \ref{defn-moduli}, we say an lc-trivial morphism $f : (X, B) \to Y$ is \emph{good} if the equation $\widetilde{\mathbf{M}}=\gamma^* \mathbf{M}$ holds. In particular, every lc-trivial fibration is good by definition.
\end{defn}

\begin{lem}\label{lem-finite-pullback}
	Let $f : X \to Y$ be a generically finite proper surjective
	morphism between normal varieties, and let $\mathbf{D}$ be an $\R$-b-Cartier divisor on $X$ with $\mathbf{D} \sim_\K 0/Y$ (resp. $\mathbf{D} \sim 0/Y$, $\mathbf{D} = 0/Y$). Then, we have 
	$$\mathbf{D} \sim_\K (\text{resp.} \sim, =)  f^*(\frac{1}{\deg f}f_*\mathbf{D}) .$$ 
\end{lem}
\begin{proof}
    Note that $f_*\mathbf{D}$ is $\R$-b-Cartier b-divisor by Lemma \ref{lem-descend} and thus the pull-back is well-defined. Replacing $X,Y$ we can assume both $\mathbf{D},f_*\mathbf{D}$ descend to $X,Y$ respectively. The rest can be argued analogously as in Remark \ref{rem-gen-pullback}.
\end{proof}

\begin{rem}\label{rem--moduli-lc-trivial-morphism}
 Note the followings:
	\begin{enumerate}
		\item 	With the notation of Definition \ref{defn-moduli}, by Lemma \ref{lem-finite-pullback}, an lc-trivial morphism being good is equivalent to $\widetilde{\mathbf{M}} = 0/Y$.
		
		\item  With the notation of Definition \ref{defn-pushdown}, if $\widetilde{\mathbf{M}} \sim_\K 0/Y$, then by Lemma \ref{lem-finite-pullback}, we see $\widetilde{\mathbf{M}} \sim_\K f^* \mathbf{M}$. In this case, $\widetilde{\mathbf{M}}$ is b-nef and (log) abundant (resp. b-semi-ample, b-ample) over $Z$ if and only if $\mathbf{M}$ is also.
	\end{enumerate}
\end{rem}

\begin{prop}\label{prop-g-finite-pullback}
	Let $f : (X, B) \to Y$ be a good lc-trivial morphism and $X \overset{\widehat{f}}{\to} \widehat{Y} \overset{\tau}{\to} Y$ be a factorisation of $f$ such that $\tau$ is generically finite. Then, $\widehat{f}: (X,B) \to \widehat{Y}$ is also a good lc-trivial morphism with
	$$\tau^*\mathbf{M} = \widehat{\mathbf{M}}$$
	where $\mathbf{M}$ and $\widehat{\mathbf{M}}$
	are the moduli b-divisors on $Y$ and $\widehat{Y}$ respectively. In particular, $\widehat{f}$ is a good lc-trivial morphism. Moreover, if $ Y$ is proper over $Z$, then $\mathbf{M},\widehat{\mathbf{M}}$ are b-nef$/Z$.
\end{prop}
\begin{proof}
	Let $\widehat{f}\!:  \!X \!\overset{\breve{f}}{\to} \!\breve{Y}\! \overset{\varrho}{\to} \!\widehat{Y}$ and $f\!:\!X\! \overset{\widetilde{f}}{\to}\! \widetilde{Y} \!\overset{\gamma}{\to}\! Y$ be the corresponding Stein factorisations respectively. Consider the following commutative diagram:
	$$
	\xymatrix{
	   X\ar[dr]_{\widetilde{f}}\ar[r]^{\breve{f}} & \breve{Y} \ar[r]^{\varrho}\ar@{-->}[d]_{\widetilde{\tau}} &   \widehat{Y} \ar[d]^{\tau} \\
		& \widetilde{Y} \ar[r]^{\gamma} & Y } 
	$$
	By the rigidity lemma, the induced map $\widetilde{\tau}$ is a morphism. Because $\varrho$ is finite and $\tau$ is generically finite, one infers that $\widetilde{\tau}$ is generically finite. Since $\widetilde{f}$ is a contraction, one infers that $\widetilde{\tau}$ is birational which in turn implies that $\widehat{f}: (X,B) \to \widehat{Y}$ is also an lc-trivial morphism. Since $f$ is good, we have $ \widetilde{\mathbf{M}}=\gamma^*\mathbf{M} $ and $\widehat{\mathbf{M}}=\frac{1}{\deg \varrho} \varrho_*\breve{\mathbf{M}}$ where $\breve{\mathbf{M}}$ and $\widetilde{\mathbf{M}}$ are the moduli b-divisors given by $\breve{f}$ and $\widetilde{f}$ respectively. Note that $\breve{\mathbf{M}}=\widetilde{\mathbf{M}}$ as $\widetilde{\tau}$ is birational. By Lemma \ref{lem-finite-pullback} we immediately obtain that $\breve{\mathbf{M}}=\varrho^*\widehat{\mathbf{M}}$ which proves the first assertion. The second assertion follows from Theorem \ref{thm--lc-trivial}.
\end{proof}

\begin{rem}[Comparison of sub-lc centres]\label{rem-compare-lc-centre} 
	With the notation of Definition \ref{defn-moduli}, as we discussed in Remark \ref{rem-pushdown}(3), if we suppose $(X,B)$ is sub-lc, then we have:
	\begin{enumerate}
		\item Every g-sub-lc centre of the induced g-pair $(Y,B_Y+M_Y)$ is dominated by some sub-lc centre.
		
		\item If $f:(X,B) \to Y$ is good, then every vertical sub-lc centre of $(X,B)$ is mapped onto a g-sub-lc centre of $(Y,B_Y+M_Y)$. 
	\end{enumerate}
\end{rem}

\subsection{Dlt models of lc-trivial morphisms.}\label{subsec-dlt-trivial-morphism}
In this subsection we introduce the notion of \emph{dlt model}. This class of morphisms can be regarded as ``dlt models" of lc-trivial morphisms.

\begin{defn}\label{defn-dlt-trivial-fib}
	An lc-trivial morphism $f : (X, B) \to Y$ is called a \emph{dlt model} if $(X,B)$ is dlt, and the induced g-pair $(Y,B_Y+M_Y)$ is g-dlt. We say a dlt model is \emph{good (resp. fibred)} if $f$ is good (resp. with connected fibres), and we say it is \emph{$\Q$-factorial (resp. (quasi-)projective) }if both $X$ and $Y$ are $\Q$-factorial (resp. (quasi-)projective).
\end{defn}		

Next we prove that every lc-trivial morphism from an lc pair possesses a dlt model, in the sense of B-birational equivalence. Let us start with the following lemma.

\begin{lem}\label{lem-dlt-trivial-fib}
	Let $f:(X,B) \to Y$ be an lc-trivial morphism from an lc pair and $(Y,B_Y+M_Y)$ be the induced g-lc pair. Let $\phi:Y' \to Y$ be a g-dlt model and $\pi:(\overline{X},\overline{B}) \to X$ be a quasi-projective log smooth model of $(X,B)$ such that $\overline{f}:\overline{X} \dashrightarrow Y'$ is a morphism. Then, $(\overline{X}/Y',\overline{B})$ has a good log minimal model $(X'/Y',B')$.
	$$
	\xymatrix{
		X \ar[d]_{f} & \overline{X} \ar[l]_{\pi}\ar[dr]_{\overline{f}}\ar@{-->}[r]^{} &  X'\ar[d]^{f'}  &\\
		Y & &Y' \ar[ll]^{\phi}   } 
	$$
	Moreover, $(X',B')$ is a log birational model of $(X,B)$ and the induced map $(X',B') \dashrightarrow (X,B)$ is a B-birational contraction.
\end{lem}
\begin{proof}
	Let $f:X \to \widetilde{Y} \to Y$ and $\overline{f}:\overline{X} \to \widetilde{Y'} \to Y'$ be the Stein factorisations. The induced map $\widetilde{\phi}:\widetilde{Y'} \to \widetilde{Y} $ is a birational morphism. Replacing $Y,Y'$ we may assume $f$ is a contraction. Note that it no longer holds that $(Y',B_{Y'}+M_{Y'})$ is g-dlt. However, by Remark \ref{rem-compare-lc-centre}(1), we see that $a(\Gamma,Y,B_Y+M_Y)=0$ for every exceptional$/Y$ prime divisor $\Gamma$ on $Y'$. By Remark \ref{rem-pushdown}(3), replacing $(\overline{X},\overline{B})$ with a suitable blow-up, we can assume every exceptional$/Y$ prime divisor $\Gamma$ on $Y'$ is dominated by a component of $S \subset\rddown{\overline{B}}$ with $a(S,X,B)=0$.
	
	Write $E:=K_{\overline{X}}+\overline{B} - \pi^* (K_X+B) \ge0$ and $E=E^h+E^v$, where $E^h$ (resp. $E^v$) denotes the horizontal (resp. vertical) part. Since $E^h$ is exceptional$/X$, we see
	$$
	E^h|_{F} = N_\sigma((K_{\overline{X}}+\overline{B})|_F)
	$$
	where $F$ is a general fibre.
	
	Since $E^v$ is exceptional$/X$, it is very exceptional$/Y$ by definition. We claim that it is also very exceptional$/Y'$. To this end, let $Z'$ be an irreducible closed subset of $\overline{f}(E^v) \subset Y'$, and let $0 \le E_{Z'}  \le  E^v$ be the effective divisor supported by those components of $E^v$ mapped onto $Z'$ with the same coefficients. Write $E^v=\sum_{Z'} E_{Z'}$. If $\mathrm{codim}_{Y'} Z' \ge 2$, then  $E_{Z'}$ is very exceptional$/Y'$. If $\mathrm{codim}_{Y'} Z'=1$ and $\mathrm{codim}_{Y}\phi(Z')=1$, then it is also very exceptional$/Y'$ as $E_{Z'}$ is very exceptional$/\phi(Z')$. It remains to prove when $Z'$ is an exceptional$/Y$ divisor. Recall our assumption that $Z'$ is dominated by a component of $S \subset\rddown{\overline{B}}$ with $a(S,X,B)=0$. Obviously $S \notin \Supp E_{Z'}$. Therefore $E_{Z'}$ is again very exceptional$/Y'$. Since $E^v$ is the sum of $E_{Z'}$'s, the claim is proved. 
	
	Hence, by Lemma \ref{lem-exc-2+}, we can run an MMP$/Y'$ on $K_{\overline{X}}+\overline{B}$ which terminates with a good log minimal model. We thus obtain the first assertion. Moreover, since the divisor contracted by the above MMP is exactly $E$, one infers that $(X',B')$ is a log birational model of $(X,B)$, and the induced map $(X',B') \dashrightarrow (X,B)$ is a B-birational contraction.
\end{proof}

\begin{rem}\label{rem-can-bundle}
	Notation as above, we remark the followings.
	\begin{enumerate}
		\item If $f$ is good, then every vertical lc centre of $(X,B)$ is mapped onto a g-lc centre of $(Y,B_{Y}+M_{Y})$ by Remark \ref{rem-compare-lc-centre}(2). However, a g-lc centre of $(Y,B_{Y}+M_{Y})$ is not necessarily an lc centre of $(Y,B_{Y})$ unless it is g-dlt. 
		
		\item One should not expect in general that the rational map $X' \dashrightarrow X$ is a morphism as the example below indicates.
	\end{enumerate}
\end{rem}

\begin{exa}
	Let $(X,B)$ be a log terminal pair and the diagram below be a $K_X$-flip.
	$$
	\xymatrix{
		X \ar[dr]_{f} \ar@{-->}[rr]^{\phi}&   &  X^+\ar[dl]^{f^+}\  &\\
		& Z &    } 
	$$
	Suppose $\phi$ is also a $K_X+B$-flop. Then, for any common crepant model $(W,B_W)$ of $(X,B)$ and $(X^+,B^+)$, we have $(W,B_W)$ is sub-klt but not klt.
\end{exa}

\begin{prop}[Existence of dlt models]\label{prop-dlt-trivial-fib}
	Let $f:(X,B) \to Y$ be an lc-trivial morphism from an lc pair. Then, there exist a birational model $\phi:Y' \to Y$ and a log birational model $(X',B')$ of $(X,B)$ with a B-birational contraction $\pi:(X',B') \dashrightarrow (X,B)$ such that the induced map $f':(X',B')\dashrightarrow Y'$ is a quasi-projective $\Q$-factorial dlt model.
	$$
	\xymatrix{
		X' \ar[d]_{f'} \ar@{-->}[r]^{\pi}   &  X\ar[d]^{f}\  &\\
		Y' \ar[r]^{\phi} &    Y } 
	$$
	In particular, if $f:(X,B) \to Y$ is good, then $f':(X',B') \to  Y'$ is good.
\end{prop}
\begin{proof}
	The proposition follows directly from Lemma \ref{lem-g-dlt-blow-up} and Lemma \ref{lem-dlt-trivial-fib}. 
\end{proof}

\subsection{Adjunction for fibre space commutes with restriction.}\label{subsec-adj-commute}
The main goal of the rest subsections is to prove the following theorem which is a generalisation of \cite[Theorem 1.1]{fujino-gongyo2}. 

\begin{thm}\label{thm-abun-moduli}
	Let $f : (X,B) \to Y$ be a good lc-trivial morphism from an lc pair to a complete variety. Then, the moduli b-divisor $\mathbf{M}$ is b-nef and log abundant with respect to the induced g-pair.
\end{thm}

We begin with elementary lemmas.

\begin{lem}[Bertini]\label{lem-Bertini}
	Let $T$ be a smooth prime divisor of a smooth variety $Y$ and $Q \subset T$ be a smooth prime divisor of $H$. Then, there exists a smooth prime divisor $P$ containing $Q$ such that $(Y,P+T)$ is log smooth.
\end{lem}
\begin{proof}
	 We can assume $Y$ is projective. Let $\pi:Y' \to Y$ be the blow-up at $Q$ and apply a theorem of Bertini \cite[Corollary 10.9]{hartshorne} to $Y'$.
\end{proof}

\begin{lem}\label{lem-log-abundance}
	Let $f:X \to Y$ be a contraction of normal varieties. Suppose the following conditions holds:
	\begin{enumerate}
		\item (Weakly semi-stable) There exist reduced divisors $\Delta,\Delta_Y$ on $X,Y$ respectively, such that $f:(X,\Delta) \to (Y,\Delta_Y)$ satisfies the conditions listed in Theorem \ref{thm-ss-reduction}.
		
		\item (Lc-trivial) There exists a sub-boundary $B$ such that $f:(X,B) \to Y$ is an lc-trivial fibration.
		
		\item (Weakly resolved) Denote by $\mathbf{B}$ the discriminant b-divisor and by $\mathbf{M}$ the moduli b-divisor with traces $B_Y,M_Y$ on $Y$. We have that, $B$ is supported by $\Delta$, $B_Y$ is supported by $\Delta$, and $M_Y$ represents $\mathbf{M}$.
	\end{enumerate}
    Then, we have:
    \begin{enumerate}
    	\item[(i)] Given a prime divisor $Q$ of $Y$ such that $(Y,Q+\Delta_Y)$ is log smooth, if $Q \nsubseteq \Delta_Y$, then $b_Q=1$ and $(X,B+f^*Q)$ is sub-lc (for definition of $b_Q$, see Section \ref{subsec-lc-trivial-fib}).
    	
    	\item[(ii)] Given a horizontal prime divisor $S\subset B^{=1}$ and any prime divisor $Q$ of $Y$, if $Q \nsubseteq \Delta_Y$, then $b_Q=1$ with respect to $f|_S:(S,B_S) \to Y$, where $K_S+B_S=(K_X+B)|_S$.
    	
    	\item[(iii)] Given a vertical prime divisor $S\subset B^{=1}$ with its image $T$ a prime divisor of $Y$, and any prime divisor $Q$ of $T$, if $Q \nsubseteq \Delta_T=(\Delta_Y -T)|_T$, then $b_Q=1$ with respect to $f|_S:(S,B_S) \to T$, where $K_S+B_S=(K_X+B)|_S$.
    \end{enumerate}
\end{lem}                                                                                                                    
\begin{proof}
	(i). Set $b:=\sup \{t|(X,B+tf^*Q)\text{ is sub-lc}\}$. By condition (3), $f^*Q$ does not contain any sub-lc centre in its support, hence $b>0$. Let $S$ be the newly constructed sub-lc centre of $(X,B+bf^*Q)$. Since $Y$ is an Ambro model, the image $T$ of $S$ is a sub-lc center of $(Y,B_Y+bQ)$. Because $Q$ does not contain any sub-lc centre of $(Y,B_Y)$ in its support, we deduce that $T$ is also newly constructed, hence $b=1$. In particular, we have $b_Q=1$.
	
	(ii). Shrinking $Y$ and $X$ accordingly, we may assume $\Delta_Y=0$. This follows from (i).
	
	(iii). Shrinking $Y$ and $X$ accordingly, we may assume $Q$ is smooth and $\Delta_Y=T$. By Lemma \ref{lem-Bertini}, there is a prime divisor $P \subset Y$ containing $Q$ such that $(Y,P+\Delta_Y)$ is log smooth. So, by (i), we deduce $(X,\Delta+f^*P)$ is sub-lc which in turn implies $b_Q=1$.
\end{proof}

\begin{rem}
	Lemma \ref{lem-log-abundance} still holds if Condition (1) is weakened to equidimensional, that is, $f$ satisfies the conditions listed in Theorem \ref{thm-equidimension}.
\end{rem}

A similar technique from \cite[Theorem 1.1]{fujino-gongyo2} will be used in the next lemma. But we will not directly apply \cite[Theorem 1.1]{fujino-gongyo2}. See also \cite[Proposition 4.4]{floris-lazic}. 

\begin{lem}\label{lem-moduli-res}
	Let $f,(X,B),\Delta,(Y,B_Y+M_Y),\Delta_Y$ be as in Lemma \ref{lem-log-abundance}, and $S\subset B^{=1}$ be a prime divisor of $X$. Suppose further that: 
    \begin{enumerate}
	\item The image $T$ of $S$ is $Y$ or a prime divisor of $Y$.
	
	\item $-B^{<0}|_{F}= N_\sigma((K_{X}+B^{>0})|_{F})$ on a general fibre $F$ of $f$.
    \end{enumerate} 
    If we denote by $N_T$ the pre-moduli divisor of the restriction morphism $f|_S:(S,B_S) \to T$, where $K_S+B_S=(K_X+B)|_S$, then we have
	$$
	M_Y|_T=N_T.
	$$
\end{lem}
\begin{proof}
	By construction we have 
	$$(B + \sum_{P \in \Delta_{Y}} b_P f^*P)^{>0} \le \Delta.$$
	Here by the notation $P \in \Delta_{Y}$ we mean $P$ is a component of the reduced divisor $\Delta_{Y}$, we see $(X,(B + \sum_{P \in \Delta_{Y}} b_P f^*P)^{>0})$ is sub-lc. Now pick a sufficiently small number $\epsilon>0$, write $\Theta:=B+ \sum_{P \in \Delta_{Y}} b_P f^*P$ and run an MMP$/Y$ on 
	$$
	K_{X} +\Xi :=K_{X}+\Theta^{>0} +\epsilon \sum_i E_i 
	$$ 
	where $E_i's$ are all prime divisors on $X$ with $f(E_i) \in \Delta_{Y}$ and the multiplicities $\mult_{E_i}\Theta^{>0} <1$. By Condition (2), the horizontal divisor $D^h:=-B^{h,<0} \ge 0$ satisfies $$D^h|_{F}=-B^{h,<0}|_{F}=-B^{<0}|_{F}= N_\sigma((K_{X}+\Xi)|_{F})$$
	where $B^{h}$ denotes the horizontal part of $B$. Moreover, since $\epsilon$ is sufficiently small, we see $\Xi \le \Delta$ which in turn implies that $(X,\Xi)$ is lc. One can easily check that the vertical divisor 
	$$
	D^v:=-\Theta^{v,<0} +\epsilon\sum_i E_i \ge 0$$ 
	is very exceptional$/Y$. Now we have
	$$
	K_{X} +\Xi \sim_\R D/Y
	$$
	where $D:=D^h+D^v$. Therefore, by Lemma \ref{lem-exc-2+} the above MMP$/Y$ contracts all components of $D$ and terminates with a good minimal model $(X'/Y,\Xi')$ where $\Xi'$ is the birational transform of $\Xi$. Consider the commutative diagram where $f'$ denotes the induced morphism.
	$$
	
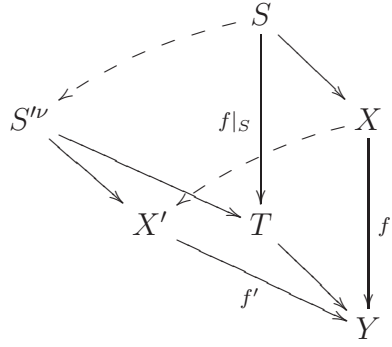
\xymatrix{
		&  &	S \ar@{-->}@/_/[dll]^{ }  \ar[dr]_{}\ar[dd]_{f|_{S}}  \\
		S'^\nu \ar[dr]_{}\ar[drr]_{} &  &	&  X\ar[dd]^{f}\ar@{-->}@/_/[dll]^{ }    \\
		&X'\ar[drr]_{f'}  &	T\ar[dr]_{}   \\
		&  &	&Y 
	} 
	$$
	Note that the birational transform $\Xi'=B'+ \sum_{P \in \Delta_{Y}} b_P f'^*P$, and that the rational map given by the MMP $(X,\Theta) \dashrightarrow (X',\Xi')$ is B-birational. Since $K_{X'}+B' \sim_\R 0/Y$, the pair $(X',\Xi')$ is $\Q$-factorial lc and the vertical part $\Xi'^v$ is a reduced divisor. In addition, since $f$ is weakly semi-stable, it holds that
	$$
	\Xi'^v=(B'+ \sum_{P \in \Delta_{Y}} b_P f'^*P)^v=  \sum_{P \in \Delta_{Y}} f'^*P=f'^*\Delta_{Y}.
	$$
	Because $\mult_{S}B=1$, the rational map $X \dashrightarrow X'$ does not contract $S$. Let $S'$ be the birational transform with its normalisation $\nu:S'^\nu \to S'$. Therefore, by the adjunction formula, one infers that 
	\begin{equation*}\tag{$\spadesuit$}
	\Xi_{S'^\nu} \ge (B'+ \sum_{P \in \Delta_{Y}} b_P f'^*P - S')|_{S'^\nu}\ge  \sum_{Q \in \Delta_{T}} f_{S'}^* Q=f_{S'}^*\Delta_{T}
	\end{equation*}
	where $f_{S'}:=(f' \circ \nu)|_{S'^\nu}:S'^\nu \to T$, $K_{S'^\nu}+\Xi_{S'^\nu}=(K_{X'}+\Xi')|_{S'^\nu}$, $\Delta_{T}=\Delta_Y$ when $T=Y$, and $\Delta_{T}:=(\Delta_{Y}-T)|_{T}$ is a snc divisor on $T$ when $T$ is a divisor. 
	
	Writing $C_{T}$ for the discriminant divisor of $f|_{S}:(S,B_{S})\to T$, since $f|_{S}$ has equidimensional and reduced fibres, by Lemma \ref{lem-log-abundance}, we have $\Supp C_{T} \subseteq \Delta_{T} $, hence $C_{T}= \sum_{Q \in \Delta_{T}} (1-b_Q)Q$. Since  $(S,B_{S}+\sum_{Q \in \Delta_{T}} b_Q' (f|_{S})^* Q)$ is sub-lc, where $b_Q':=b_P$ if $Q$ is an irreducible component of $P|_{T}$ for some $P \in \Delta_{Y}$, we deduce $b_Q \ge b_Q'$, hence $C_{T} \le B_T$ where $B_T=B_Y$ if $T=Y$ and $B_T=(B_{Y}-T)|_{T}$ if $T$ is a divisor.
	
	On the other hand, writing $K_{S'}+B_{S'^\nu}=(K_{X'}+B')|_{S'^\nu}$, because $$\Xi_{S'^\nu}=B_{S'^\nu}+\sum_{Q \in \Delta_{T}} b_Q' f_{S'}^* Q,$$
	by ($\spadesuit$) it in turn holds that 
	$$B_{S'^\nu} \ge  \sum_{Q \in \Delta_{T}} (1-b_Q') f_{S'}^* Q.$$ 
	We therefore deduce $b_Q \le b_Q'$, hence $C_{T} \ge B_T$. So we conclude $N_{T} = M_{Y}|_{T}$.
\end{proof}

\begin{defn}[Filtration of strata in codimension one]\label{defn-filtration-codim-one}
	Let $(X,B)$ be a sub-dlt pair and $S$ be a stratum. An ascending sequence of strata 
	$$S=S_0 \subset S_1 \subset \ldots \subset S_n =X 
	$$ is called a \emph{filtration of strata in codimension one} if the codimension of $S_i$ in $S_{i+1}$ is one for all $i$. 
\end{defn}

\begin{defn}\label{defn-saturated-log-morphism}
	Let $f:(X,B) \to (Y,B_Y)$ be a surjective morphism of sub-dlt pairs, and $T$ be a stratum of $(Y,B_Y)$. A stratum $S$ of $(X,B)$ is said to be \emph{saturated over $T$} if there exist a filtration of strata in codimension one of $T$: 
		$$
		T=T_0 \subset T_1 \subset \ldots \subset T_m =Y,
		$$
		and a filtration of strata in codimension one of $S$:
		$$S=S_0 \subset S_1 \subset \ldots \subset S_n =X 
		$$
		such that $f(S)=T$ and for every $0 \le j \le m$, there is some $j \le i \le n$ with $f(S_i)=T_j$.	
\end{defn}

Recall the following fact from scheme theory:

{\noindent \textbf{Fact} $(\clubsuit)$:} Let $X,Y$ and $S$ be integral schemes of finite Krull dimension, $f:X \to S$ be an integral morphism and $g:Y\to S$ be a morphism which factors through $f$. Then, the induced map $Y \to X \times_S Y$ is a closed immersion since $f$ is separated. Because an integral morphism is stable under base change and preserves Krull dimensions (\cite[\href{https://stacks.math.columbia.edu/tag/0ECG}{Lemma 0ECG}]{stacks-project}), the closed immersion maps $Y$ \emph{set-theoretically} to a component of $X \times_S Y$. Moreover, it also maps \emph{scheme-theoretically} when $f$ has reduced geometric generic fibre, for example, $f$ is dominant with $k(X)/k(S)$ separable, since geometric reducibility is stable under base change (\cite[\href{https://stacks.math.columbia.edu/tag/0576}{Lemma 0576}]{stacks-project}).
 
\begin{thm}\label{thm-moduli-div}
    Let $f:(X,B) \to Y$ be a good dlt model (see Definition \ref{defn-dlt-trivial-fib}) and $T$ be a stratum of the induced dlt pair $(Y,B_Y)$. Suppose $S$ is a stratum of $(X,B)$ saturated over $T$. Then, 
    \begin{enumerate}
    	\item $f|_S:(S,B_S) \to T$ is a good dlt model where $K_S+B_S=(K_X+B)|_S$.
    	
    	\item If we denote by $(T,B_T+M_T)$ the g-dlt pair with data $\mathbf{M}|_T$ given by the adjunction formula $K_T+B_T+M_T=(K_Y+B_Y+M_Y)|_T$, and we denote by $(T,C_T +N_T)$ the g-dlt pair with data $\mathbf{N}$ given by the dlt model $f|_S:(S,B_S) \to T$. Then, $$\mathbf{M}|_T = \mathbf{N}.$$
    \end{enumerate}  
\end{thm}

\begin{proof}
	We will prove by induction on the dimension of $X$. If $\dim X=1$, then there is nothing to prove. Suppose $\dim X=n$ and by induction we may assume Theorem \ref{thm-moduli-div} holds in dimension $\le n-1$. Since $S$ is saturated over $T$, by inductive assumption, replacing $S,T$, we can assume $S$ is a prime divisor on $X$ and $T=Y$ or $T$ is a prime divisor on $Y$. We can assume further: 
	
	{\noindent \textbf{(I). $f$ has connected fibres.}} Indeed, let $f:X \overset{\widetilde{f}}{\to} \widetilde{Y} \overset{\gamma}{\to} Y$ be the Stein factorisation which in turn induces a factorisation of $f|_S:S\overset{\widehat{f_S}}{\to} \widehat{T}^\nu \overset{\gamma|_{\widehat{T}^\nu}}{\to} T $ where $\widehat{T}$ is the image of $S$ on $\widetilde{Y}$ with is normalisation $\widehat{T}^\nu$. Since $\gamma|_{\widehat{T}^\nu}$ is finite and $f|_S$ is an lc-trivial morphism, we deduce that $\widehat{f_S}$ is an lc-trivial morphism. Fixing a sufficiently general $\R$-rational function $\varphi$, we may assume $\varphi|_S$ is well-defined. By construction we have $\mathbf{N}=\frac{1}{\deg \gamma_T} \gamma_{T,*}\widehat{\mathbf{N}}$ where $\gamma_T:=\gamma|_{\widehat{T}^\nu}$. Thus, by Lemma \ref{lem-finite-pullback}, it suffices to prove $\widetilde{\mathbf{M}}|_{\widehat{T}^\nu} = \widehat{\mathbf{N}}$ and $\widehat{f_S}$ is good. Thanks to Proposition \ref{prop-dlt-trivial-fib} there exists a dlt model of $\widetilde{f}:(X,B) \to \widetilde{Y}$. Replacing $X,Y,f$ with the previous dlt model and the rest accordingly we may assume $f$ has connected fibres.

    {\noindent \textbf{(II). $f|_S$ has connected fibres.}}  Indeed, let $f|_S:S \overset{\widetilde{f_S}}{\to} \widetilde{T} \overset{\zeta}{\to} T$ be the Stein factorisation. Since $\zeta$ is finite, by Theorem \ref{thm-extend-finite-cover}, there exists a finite morphism $\rho: Y' \to Y$ together with a prime divisor $\widehat{T}$ such that $\nu:T':=\widetilde{T} \to \widehat{T}$ is the normalisation. Let $f':(X',B') \to Y'$ be the induced lc-trivial fibration and let $\varrho: X' \to X$ be the induced morphism. By Fact $(\clubsuit)$, there is a component $S'$ of $S \times_T T'$ which is isomorphic to $S$. Consider the following diagram, where $\varrho_{S}:=\varrho|_{S'}$, $f_{S}':=f'|_{S'}$ and $\rho_{T}:=\rho|_{T'}=\zeta$. Hence $\varrho_{S}$ is the identity and $f_{S}'=\widetilde{f_S}$ is a contraction.
	 $$
	\xymatrix{
		S'\ar[dr]_{}\ar[dd]_{f_{S}'}\ar[rr]^{\varrho_{S}} && S \ar[dr]_{}\ar[dd]_{}\\
		&  X'\ar[dd]^{ }   \ar[rr]^{} & & X \ar[dd]^{f} \\
		T'\ar[dr]_{\nu}\ar[rr]^{\ } && T \ar[dr]^{} \\
		&Y' \ar[rr]^{\rho}   &  & Y
	} 
	$$
	Let $\mathbf{M}',\mathbf{N}'$ be the moduli b-divisors defined by $f',f_{S}'$ respectively. By Lemma \ref{lem-lc-fib-base-change}, we infer that $\mathbf{M}'=\rho^*\mathbf{M}$ and $\frac{1}{\deg \zeta}\zeta_* \mathbf{N}'=\mathbf{N}$. So, by Lemma \ref{lem-finite-pullback}, it is enough to prove $\mathbf{M}'|_{T'}=\mathbf{N}'$ which implies the goodness of $f|_S$. Applying Proposition \ref{prop-dlt-trivial-fib} and replacing we may assume $f|_S$ has connected fibres.

   {\noindent \textbf{Apply weakly semi-stable reduction.}} Finally, we let $Z=\Supp B$, $Z_Y=\Supp B_Y$ and apply Theorems \ref{thm-equidimension} and \ref{thm-ss-reduction} to construct a commutative diagram so that $\mathbf{M}$ descends to $Y_1$ and $\mathbf{N}$ descends to $T_1$, where $T_1$ is the birational transform of $T$ on $Y_1$.
    $$
    \xymatrix{
    	(X_2,\Delta_2) \ar[d]_{f_2} \ar[r]^{\mu}	& (X_1,\Delta_1) \ar[d]_{f_1} \ar[r]^{\pi}   &  X\ar[d]^{f}\  &\\
    	(Y_2,\Delta_{Y_2}) \ar[r]^{\eta}	& (Y_1,\Delta_{Y_1}) \ar[r]^{\phi} &    Y } 
    $$  
    Let $K_{X_2}+B_2=(\pi \circ \mu)^*(K_X+B)$. By Lemma \ref{lem-lc-fib-base-change}, we deduce $\mathbf{M}_2 = \eta^* \mathbf{M}$ and $\mathbf{N}_2 = (\eta|_{T_2})^* \mathbf{N}$ where $\mathbf{M}_2$ and $\mathbf{N}_2$ are the moduli b-divisors given by $f_2:(X_2,B_2) \to Y_2$ and $f_2|_{S_2}:(S_2,B_{S_2}) \to T_2$ respectively, and hence $\mathbf{M}_2|_{T_2} = (\eta|_{T_2})^* \mathbf{M}|_{T_1}$. Hence, $N_{T_2} =(\eta_{T_2})^* N_{T_1}$ represents $\mathbf{N}_2$ and $M_{Y_2} = \eta^* M_{Y_1}$ represents $\mathbf{M}_2$. So it suffices to show that $N_{T_2} = M_{Y_2}|_{T_2}$, which follows from Lemma \ref{lem-moduli-res}.
\end{proof}

\subsection{Log abundance of the moduli b-divisor.}\label{subsec-log-abundance}
In this subsection we prove Theorem \ref{thm-abun-moduli} and its corollaries.

\begin{lem}\label{lem-sat-stratum}
	Given a good dlt model $f:(X,B) \to Y$ of an lc-trivial morphism, for any stratum $T$ of $(Y,B_Y)$, there exists a stratum $S$ of $(X,B)$ saturated over $T$.
\end{lem}
\begin{proof}
	By induction on the codimension of $T$ and Theorem \ref{thm-moduli-div}(1).
\end{proof}

\begin{proof}[Proof of Theorem \ref{thm-abun-moduli}]
    By Proposition \ref{prop-dlt-trivial-fib}, replacing $f:(X,B) \to Y$ we may assume it is a good dlt model. Let $(Y,B_Y+M_Y)$ be the induced g-dlt pair. It suffices to show that, for every stratum $T$ of $(Y,B_Y)$, the moduli b-divisor $\mathbf{M}|_T$ is b-nef and abundant. Thanks to Theorem \ref{thm-moduli-div} and the above lemma, there exists a stratum $S$ of $(X,B)$ with $f(S)=T$ such that, writing $K_S+B_S=(K_X+B)|_S$, the morphism $f|_S:(S,B_S) \to T$ is a good dlt model and $\mathbf{M}|_T= \mathbf{N}$ where $\mathbf{N}$ is the moduli b-divisor of $f|_S$. So, it is sufficient to prove $\mathbf{M}$ is b-nef and abundant.
    
    To this end, let $W$ be a minimal horizontal lc centre of $(X,B)$. Again by Theorem \ref{thm-moduli-div}, writing $K_W+B_W=(K_X+B)|_W$, the morphism $f|_W:(W,B_W) \to Y$ is a good dlt model and $\mathbf{M}= \mathbf{P}$ where $\mathbf{P}$ is the moduli b-divisor of $f|_W$. Replacing $(X,B)$ with $(W,B_W)$ we may assume $\rddown{B}$ has no horizontal$/Y$ component. Finally, the result follows from Corollary \ref{cor--klt-trivial} and Remark \ref{rem--moduli-lc-trivial-morphism}.
\end{proof}

\begin{lem}[Compactification of dlt model]\label{lem-compact}
	Let $f : (X,B) \to Y$ be a fibred dlt model. Suppose $g:Y \to Z$ is a proper morphism. Then, there exists an open immersion $Z  \subset Z^c$ and a projective fibred dlt model $f^c: (X^c,B^c) \to Y^c$ with $g^c:Y^c \to Z^c$ proper such that $(X^c|_{(g\circ f)^{-1}Z}, B^c|_{(g\circ f)^{-1}Z})=(X,B)$, $f^c|_{(g\circ f)^{-1}Z}=f$ and $Y^c|_{g^{-1}Z}=Y$.
\end{lem}
\begin{proof}
	This follows from \cite[Theorem 1.2]{has-mmp} (see also \cite[Theorem 3.2]{hu}).
\end{proof}

The following theorem is a relative version of Theorem \ref{thm-abun-moduli}.
\begin{thm}\label{thm-relative-log-abundance}
	Let $f : (X,B) \to Y$ be a good lc-trivial morphism from an lc pair, $Y$ be proper over a variety $Z$. Then, the moduli b-divisor $\mathbf{M}$ is b-nef and log abundant over $Z$ with respect to the induced g-pair.
\end{thm}
\begin{proof}
	Since the question is local (see Definition \ref{defn--abund}), we may assume $Z$ is affine. By Remark \ref{rem--moduli-lc-trivial-morphism}, replacing $Y$ we can assume $f$ is a contraction. By Proposition \ref{prop-dlt-trivial-fib}, replacing $f:(X,B) \to Y$ we can assume it is a fibred dlt model. By Lemma \ref{lem-compact}, we can assume $Z$ is complete and $X,Y$ are projective. By Theorem \ref{thm-abun-moduli}, $\mathbf{M}$ is b-nef and log abundant. Therefore, by Lemma \ref{lem-global-local-nef-abun}, it is b-nef and log abundant over $Z$ with respect to the induced g-pair. 
\end{proof}

\appendix

\section{Extending a finite cover from a subvariety}\label{sec3}
The main purpose of this section is to prove the following theorem on extending a finite cover over a closed subvariety to a finite cover of the same degree over the whole variety. Although in this paper, we will only use a special case the ground field is an algebraically closed field of characteristic zero, we prove it in full generality. Throughout this section, all varieties are over an arbitrary field. 

\begin{thm}\label{thm-extend-finite-cover}
	Let $X$ be a normal variety and $S$ be a closed subvariety. Suppose we are given a finite morphism $\gamma: \widetilde{S} \to S$ from a normal variety. Suppose further that one of the following conidtions holds:
	\begin{enumerate}
		\item The residue field $\kappa(\eta_S)$ at the generic point $\eta_S$ of $S$ is perfect.
		
		\item $X$ is regular at the generic point of $S$.
	\end{enumerate}
	Then there exists a finite morphism $\rho: \widetilde{X} \to X$ of normal varieties together with a closed subvariety $\widehat{S} \subset \widetilde{X}$ satisfying:
	$$
	\xymatrix{
		\widetilde{S} \ar[dr]_{\gamma}\ar[r]^{\nu} & \widehat{S} \ar[d]^{} \ar@{^(->}[r]^{}   &  \widetilde{X}\ar[d]^{\rho}\  &\\
		& S \ar@{^(->}[r]^{} &    X } 
	$$ 
	
	(1). $\widehat{S}$ is mapped onto $S$ through $\rho$ and the above diagram commutes.
	
	(2). $\nu$ is the normalisation of $\widehat{S}$. 
	
	(3). $\deg \rho= \deg \gamma$. 
\end{thm}

We need some knowledge from commutative algebra. We begin with the following elementary lemma. For a domain $A$, we denote by $\kappa(\mathfrak{p})$ the residue field at $\mathfrak{p}$, and by $K(A)$ the field of fractions.
\begin{lem}\label{lem-com-alg-1}
	Let $A \subset B$ be an integral extension of domains, where $A$ is integrally closed. Let $\mathfrak{q} \subset B$ be a prime ideal, $\mathfrak{p}=\mathfrak{q} \bigcap A $, and $A/\mathfrak{p} \subset B/\mathfrak{q}$ be the induced inclusion of domains. Suppose further that one of the following conditions holds:
	\begin{enumerate}
		\item $\kappa(\mathfrak{q})$ is generated by one element over $\kappa(\mathfrak{p})$. 
		
		\item $A$ is Noetherian, and $B$ is finite and locally flat over $A$ at $\mathfrak{p}$.  
		
		\item $A$ is regular at $\mathfrak{p}$, and $[K(B):K(A)]< +\infty$. 
	\end{enumerate}
	Then, we have $[\kappa(\mathfrak{q}):\kappa(\mathfrak{p})]< +\infty$ and the inequality $$[\kappa(\mathfrak{q}):\kappa(\mathfrak{p})] \le [K(B):K(A)].$$
\end{lem}
\begin{proof}
	Let $S=A \backslash \mathfrak{p}$ be the multiplicative set. Replacing $A,B$ with $S^{-1}A,S^{-1}B$, we may assume $A$ is a local domain. In particular, $A/\mathfrak{p}$ is a field, denote by $\Bbbk$, and $B/\mathfrak{q}$ is integral over $\Bbbk$ which is also a field by \cite[Proposition 5.7]{am}. 
	
	\noindent \textbf{Case 1.} We suppose that $ B/\mathfrak{q}$ is generated by one element $\overline{b}$ over $\Bbbk$. Lift $\overline{b}$ to any element $b \in B$. Let $\alpha(X) \in K(A)[X]$ be the minimal polynomial of $b \in K(B)$ over $K(A)$. Since $A$ is integrally closed, by \cite[Proposition 5.15]{am}, we deduce $\alpha(X) \in A[X]$. Thus we have
	$K(A)\subset K(A)[b]\subset K(B)$ with $[K(A)[b]:K(A)]= \deg (\alpha)$. On the other hand, the reduction mod $\mathfrak{p}$ gives $\overline{\alpha}(X) \!\in\! A/\mathfrak{p}[X]$ which still satisfies $\overline{\alpha}(\overline{b})=0$ in $B/\mathfrak{q}$. In particular, if $\beta(X) \in \Bbbk[X]$ is the minimal polynomial of $\overline{b}$, then $\overline{\alpha}(X)$ is divisible by $\beta(X)$, hence $\deg(\beta) \le \deg (\overline{\alpha}) =\deg(\alpha)$. Therefore, we deduce $$[K(B):K(A)] \ge [K(A)[b]:K(A)] \ge  [B/\mathfrak{q}:\Bbbk].$$
	
	\noindent \textbf{Case 2.} We suppose that $A$ is Noetherian and $B$ is finite and flat over $A$. Since $A$ is Noetherian and $B$ is finite over $A$, by \cite[\href{https://stacks.math.columbia.edu/tag/00FP}{Lemma 00FP}]{stacks-project}, we deduce $B$ is of finite presentation over $A$. Moreover, because $B$ is flat over $A$, by \cite[\href{https://stacks.math.columbia.edu/tag/00NX}{Lemma 00NX}]{stacks-project}, $B$ is a free $A$-module, and $\dim_{\kappa(\mathfrak{p})} B \otimes_A \kappa(\mathfrak{p})= \dim_{K(A)} B \otimes_A K(A)$. Since $B$ is integral over $A$, by \cite[Propositions 3.5 and 5.7]{am}, we have $B \otimes_A K(A)=K(B)$. On the other hand, $\dim_{\kappa(\mathfrak{p})} B \otimes_A \kappa(\mathfrak{p})=\dim_{\kappa(\mathfrak{p})} B/\mathfrak{p}B \ge \dim_{\kappa(\mathfrak{p})} B/\mathfrak{q}$. We established the inequality.
	
	\noindent \textbf{Case 3.} We suppose that $A$ is regular at $\mathfrak{p}$, hence a regular local ring. We prove by induction on the height of $\mathfrak{p}$. We first assume $\mathrm{ht}(\mathfrak{p})=1$. Due to \cite[Proposition 9.2]{am}, $A$ is a discrete valuation ring, and let $\pi \in A$ be a uniformiser. 
	By \cite[\href{https://stacks.math.columbia.edu/tag/031F}{Lemma 031F}]{stacks-project}, we deduce $B$ is a semi-local domain with finitely many maximal ideals $\mathfrak{q}_i$ lying over $\mathfrak{p}$. Moreover, $[\kappa(\mathfrak{q}_i):\kappa(\mathfrak{p})]< +\infty$ for each $i$. Let $M=B/\pi B$. We claim that $M$ is a $B$-module of finite length. Indeed, since $B \subset K(A)^{\oplus n}$, where $n=[K(B):K(A)]$, we apply \cite[\href{https://stacks.math.columbia.edu/tag/00PE}{Lemma 00PE}]{stacks-project} to obtain that
	\begin{equation}\label{eq1}
	\mathrm{length}_A M \le n \cdot \mathrm{length}_A (A/\pi A)=n
	\end{equation}
	is finite which proves the claim by  \cite[\href{https://stacks.math.columbia.edu/tag/00IX}{Lemma 00IX}]{stacks-project}. Hence, applying \cite[\href{https://stacks.math.columbia.edu/tag/02M0}{Lemma 02M0}]{stacks-project}, we immediately deduce
	\begin{equation}\label{eq2}
	\mathrm{length}_A M = \sum_i [\kappa(\mathfrak{q}_i):\kappa(\mathfrak{p})] \mathrm{length}_{B_{\mathfrak{q}_i}} M_{\mathfrak{q}_i} \ge  [\kappa(\mathfrak{q}):\kappa(\mathfrak{p})].
	\end{equation}
	Combining (\ref{eq1}) and (\ref{eq2}), we conclude the result.
	
	Now we suppose $\mathrm{ht}(\mathfrak{p})=d$, and Lemma \ref{lem-com-alg-1} holds in height $\le d-1$. Given a regular sequence $\{x_1,\ldots,x_d\}$ of $A$, by \cite[\href{https://stacks.math.columbia.edu/tag/00NQ}{Lemma 00NQ}]{stacks-project}, $A/(x_1)$ is a regular local ring. In particular, it is a domain by \cite[\href{https://stacks.math.columbia.edu/tag/00NP}{Lemma 00NP}]{stacks-project}, which in turn implies that $(x_1)$ is prime. So, by \cite[Theorem 5.10]{am}, there is a prime ideal $\mathfrak{r} \subset B$ such that $\mathfrak{r} \bigcap A=(x_1)$. Apply the inductive assumption to $A/(x_1) \subset B/\mathfrak{r}$ to deduce $[\kappa(\mathfrak{r}):\kappa((x_1))] \le [K(B):K(A)]$ and $[\kappa(\mathfrak{q}):\kappa(\mathfrak{p})] \le [\kappa(\mathfrak{r}):\kappa((x_1))]$, hence $[\kappa(\mathfrak{q}):\kappa(\mathfrak{p})] \le [K(B):K(A)]$.
\end{proof}

\begin{rem}\label{rem-algebra-1}
	Notation as above, from equations (\ref{eq1}) and (\ref{eq2}), and by the inductive argument, we obtain a stronger inequality $\sum_i [\kappa(\mathfrak{q}_i):\kappa(\mathfrak{p})] \le [K(B):K(A)]$ in Case 3. If the equality holds, then for each $1\le i \le d$, there is a unique prime ideal $\mathfrak{r}_i \subset B$ such that $\mathfrak{r}_i \bigcap A=(x_i)$ and that $[\kappa(\mathfrak{r}_i):\kappa((x_i))]= [K(B):K(A)]$.
\end{rem}

\begin{rem}\label{rem-algebra-2}
	In the proof of Case 3, although we do not assume $B_{\mathfrak{p}}$ to be Noetherian explicitly, we can deduct it (if $\dim A_{\mathfrak{p}}=1$) by Krull-Akizuki \cite[\href{https://stacks.math.columbia.edu/tag/00PG}{Lemma 00PG}]{stacks-project}. Thus, $M$ is Artinian hence the finiteness of $\mathrm{length}_B M$. If we assume further that $B_{\mathfrak{p}}$ is integrally closed (which is reasonable by Krull-Akizuki's theorem), then $A_{\mathfrak{p}} \hookrightarrow B_{\mathfrak{q}}$ is an extension of discrete valuation rings, hence achieving a refined inequality \cite[\href{https://stacks.math.columbia.edu/tag/09E5}{Lemma 09E5}]{stacks-project}. Furthermore, assuming $K(B)/K(A)$ is finite separable, one infers the fundamental identity \cite[\href{https://stacks.math.columbia.edu/tag/09E8}{Remark 09E8}]{stacks-project}. In particular, this identity becomes a simple form as in \cite[\href{https://stacks.math.columbia.edu/tag/09EB}{Lemma 09EB}]{stacks-project} when $K(B)/K(A)$ is finite Galois. For a general form of this identity in algebraic geometry, see \cite[Examples 4.3.6 and 4.3.7]{fulton}. One can derive the inequality in Case 3 by this identity and the fact that the multiplicity $e_V Y=1$ if and only if $Y$ is regular at $V$ (see \cite[Example 4.3.5(d)]{fulton}). Finally, we remark that, a similar inequality can be established for an extension of valuation rings; see \cite[\href{https://stacks.math.columbia.edu/tag/0ASH}{Lemma 0ASH}]{stacks-project}.
\end{rem}

Recall that a ring $A$ is \emph{Japanese} if for every finite extension $L$ of its field of fractions $K$, the integral closure of $A$ in $L$ is a finitely generated $A$-module, and it is a \emph{Nagata ring} if it is Noetherian, and or every prime ideal $\mathfrak{p}$ the ring $A/\mathfrak{p}$ is Japanese. Recall a basic fact that, a Nagata ring $A$ is universally Japanese, and any finite type $A$-algebra is Nagata (see  \cite[\href{https://stacks.math.columbia.edu/tag/0334}{Proposition 0334}]{stacks-project}). Every quasi-excellent ring is a Nagata ring (see \cite[\href{https://stacks.math.columbia.edu/tag/07QV}{Lemma 07QV}]{stacks-project}), so in particular almost all Noetherian rings that occur in algebraic geometry are Nagata rings. See \cite[\href{https://stacks.math.columbia.edu/tag/0335}{Proposition 0335}]{stacks-project} for examples.

\begin{rem}\label{rem-com-alg-1}
	We collect some important cases in which the inequality of the previous lemma holds. With notation as in Lemma \ref{lem-com-alg-1}, we assume further that $A$ is Nagata. In this situation, we may replace $B$ with its integral closure in the field of fractions, so we may further assume $B$ is finite over $A$. In particular, it automatically holds that $[K(B):K(A)]< +\infty$ and  $[\kappa(\mathfrak{q}):\kappa(\mathfrak{p})]< +\infty$.
	
	If the residue field extension $\kappa(\mathfrak{q})/\kappa(\mathfrak{p})$ is separable, then by the primitive element theorem, the condition (1) is satisfied. Furthermore, if $\kappa(\mathfrak{p})$ is perfect, then $\kappa(\mathfrak{q})/\kappa(\mathfrak{p})$ is separable, and by \cite[\href{https://stacks.math.columbia.edu/tag/09H2}{Lemma 09H2}]{stacks-project}, $\kappa(\mathfrak{q})$ is perfect.
	
	If the height of $\mathfrak{p}$ is one, then $A$ is regular at $\mathfrak{p}$ as $A$ is Noetherian and integrally closed (\cite[Proposition 9.2]{am}). Note that $\mathrm{ht}(\mathfrak{q})=1$ and $B$ is integrally closed, so $B$ is regular at $\mathfrak{q}$.
\end{rem}

\begin{rem}\label{rem-com-alg-2}
	Strikingly enough, Lemma \ref{lem-com-alg-1} does not hold true if we drop all the three assumptions listed. See the following counter-example taken from \href{https://mathoverflow.net/a/198690}{Mathoverflow}.
\end{rem}

\begin{exa}
	Let $R_m=\mathbb{F}_p[t_i,x_i]_{i\in \N_m}$ where $\N_m:= \{1,2,\ldots, m\}$. Let $\sigma:R_m \to R_m$ be the order $p$ automorphism sending $t_i$ to $t_i$ and $x_i$ to $x_i+t_i$. Let $S_m \subset R_m$ be the subring of the fixed elements under the action. Then, by \cite[Proposition 7.8 and Corollary 7.7]{am}, $R_m$ and $S_m$ are Noetherian integrally closed domains and $R_m$ is integral over $S_m$. In particular, $R_m$ is the integral closure of $S_m$ in the extension of field of fractions, and the field extension has degree $p$ by Galois theory.
	
	Let $\mathfrak{q}\subset R_m$ be the prime ideal $\mathfrak{q}=(t_i)_{i \in \N_m}$. This lies over the prime ideal $\mathfrak{p}=S_m \bigcap \mathfrak{q}$. consider an element $f\in R_m$. We can write $f$ as
	$$
	f=f_0+ \sum_i t_i f_i+ \text{terms of higher orders}
	$$
	where $f_0,f_i$ are polynomials of $x_j$ and all other terms have order higher than one in $t_k$'s. Now if $f \in S$, then $\sigma(f)=f$. Note that
	$$
	\sigma(f)=f_0 +\sum_i t_i(f_i+ \partial f_0/\partial x_i ) + \text{terms of higher orders}.
	$$
	This means that $\partial f_0/\partial x_i$ is identically zero. In other words, we see that $f_0\in \mathbb{F}_p[x^p_i]$. Hence we see that $\kappa(\mathfrak{p})=K(S_m/\mathfrak{p})=\mathbb{F}_p(x^p_i)_{i \in \N_m}$. Since $\kappa(\mathfrak{q})=\mathbb{F}_p(x_i)_{i \in \N_m}$, we deduce $[\kappa(\mathfrak{q}):\kappa(\mathfrak{p})] = p^m>p=[K(R_m):K(S_m)]$.
	
	Let $R_{\infty}=\mathbb{F}_p[t_i,x_i]_{i\in \N}$. We may construct $S_{\infty}$ in the same way. Then, $R_\infty$ and $S_\infty$ are integrally closed domains and $R_\infty$ is integral over $S_\infty$, but they are not Noetherian. Let $\mathfrak{q}\subset R_\infty$ be the prime ideal $\mathfrak{q}=(t_i)_{i \in \N}$ and $\mathfrak{p}=S_\infty \bigcap \mathfrak{q}$. One can easily calculate that $[K(R_m):K(S_m)]=p$ but $[\kappa(\mathfrak{q}):\kappa(\mathfrak{p})] = +\infty$.
\end{exa}

\begin{exa}
	By making minor changes to the previous example, we may construct a counter-example, as Remark \ref{rem-com-alg-2} desired, in characteristic $0$. Let $T_m=\Z[\zeta, t_i,x_i]_{i\in \N_m}$ where $\N_m:= \{1,2,\ldots, m\}$ and $\zeta$ is the $p$-th root of unity. Let $\sigma:T_m \to T_m$ by $\sigma|_{Z[\zeta]}=1_{Z[\zeta]}$, $\sigma(t_i)=t_i$ and $\sigma(x_i)=\zeta x_i+t_i$ be an automorphism of order $p$. Let $U_m \subset T_m$ be the subring of the fixed elements under the action. Then, $T_m$ and $U_m$ are Noetherian integrally closed domains and $T_m$ is integral over $U_m$ with $[K(T_m):K(U_m)]=p$.
	
	Let $\mathfrak{q}=(1-\zeta,t_i)_{i \in \N_m}$ and $\mathfrak{p}=\mathfrak{q} \bigcap U_m$ be prime ideals. By similar arguments, one can verify that $[\kappa(\mathfrak{q}):\kappa(\mathfrak{p})] =p^m$.
\end{exa}

\begin{lem}\label{lem-com-alg-3}
	Let $A \subset B$ be a finite extension of domains, $\mathfrak{p} \subset A$ be a prime ideal and $\mathfrak{q} \subset B$ be a prime ideal such that $\mathfrak{q} \bigcap A=\mathfrak{p}$. Suppose $[K(B):K(A)]=[\kappa(\mathfrak{q}):\kappa(\mathfrak{p})] $ and $A$ is regular at $\mathfrak{p}$. Then, $B$ is regular at $\mathfrak{q}$.
\end{lem}
\begin{proof}
	Replacing $A,B$ with $A_{\mathfrak{p}},B_{\mathfrak{p}}$, we may assume $A$ is a regular local ring. By Remark \ref{rem-algebra-1}, there is a unique prime ideal $\mathfrak{q}$ lying over $\mathfrak{p}$, which in turn implies, by \cite[Corollary 5.8]{am}, $B$ is a local domain. 
	
	We prove by induction on the height of $\mathfrak{p}$. First assume $\mathrm{ht}(\mathfrak{p})=1$. Since $\dim B=1$, and by Krull-Akizuki \cite[\href{https://stacks.math.columbia.edu/tag/00PG}{Lemma 00PG}]{stacks-project}, $B$ is Noetherian, it suffices to show $B$ is integrally closed. Let $B \subset \overline{B}$ be the integral closure in the field of fractions. By Remark \ref{rem-algebra-1}, there is a unique prime ideal $\overline{\mathfrak{q}} \subset \overline{B}$ such that $\overline{\mathfrak{q}} \bigcap B=\mathfrak{q}  $. In particular, $\overline{B}$ is a discrete valuation ring. Let $\pi$ be the uniformiser of $\overline{B}$. By \cite[\href{https://stacks.math.columbia.edu/tag/09E5}{Lemma 09E5}]{stacks-project}, the ramification index of $\overline{B}$ over $A$ is one, hence $\pi \in \mathfrak{q}$. So, $\mathfrak{q}$ is principle which in turn implies $B$ is integrally closed by \cite[Proposition 9.2]{am}.
	
	Now we assume $\mathrm{ht}(\mathfrak{p})=d$ and suppose Lemma \ref{lem-com-alg-3} holds in height $\le d-1$. Let $\{x_1,\ldots,x_d\} \subset \mathfrak{p}$ be a regular sequence of $A$. By \cite[\href{https://stacks.math.columbia.edu/tag/00NQ}{Lemma 00NQ}]{stacks-project}\cite[\href{https://stacks.math.columbia.edu/tag/00NP}{Lemma 00NP}]{stacks-project}, $A/(x_1)$ is a regular local ring and $(x_1)$ is prime. By Remark \ref{rem-algebra-1}, there exists a unique prime ideal $\mathfrak{r}_1$ such that $\mathfrak{r}_1 \bigcap A=(x_1)$, and that $[\kappa(\mathfrak{r}_1):\kappa((x_1))]= [K(B):K(A)]$. Moreover, by Remark \ref{rem-algebra-1} and the discussion above, the homomorphism of discrete valuation rings $A_{(x_1)} \hookrightarrow B_{(x_1)}=B_{\mathfrak{r}_1}$ has the ramification index one, hence $\mathfrak{r}_1B_{\mathfrak{r}_1}=(x_1)$. In particular $\mathfrak{r}_1=(x_1)$. Apply the inductive assumption to $A/(x_1) \subset B/\mathfrak{r}_1$. We deduce $B/\mathfrak{r}_1$ is regular. Since $B$ is finite over $A$, we see $B$ is Noetherian. Therefore, by \cite[\href{https://stacks.math.columbia.edu/tag/00NU}{Lemma 00NU}]{stacks-project}, $B$ is a regular local ring. The lemma is proved.
\end{proof}

\begin{lem}\label{lem-com-alg-2}
	Let $A$ be an integrally closed Nagata ring, $\mathfrak{p} \subset A$ be a prime ideal and $C \supset A/\mathfrak{p}$. Suppose $C$ is an integrally closed domain finite over $A/\mathfrak{p}$. Suppose further that one of the following conditions holds:
	\begin{enumerate}
		\item The residue field $\kappa(\mathfrak{p})$ is perfect.
		
		\item $A$ is regular at $\mathfrak{p}$.
	\end{enumerate}
	Then there exists an integrally closed domain $B\supset A$ together with a prime ideal $\mathfrak{r} \subset B$ satisfying:
	$$
	\xymatrix{
		B  \ar@{->>}[r]^{q} & B/\mathfrak{r} \ar@{^(->}[r]^{i}\  &  C  &\\
		A \ar@{^(->}[u]^{}\ar@{->>}[r]^{p} &    A/\mathfrak{p} \ar@{^(->}[u]^{} \ar@{^(->}[ur]^{}  } 
	$$
	\begin{enumerate}
		\item $\mathfrak{r} \bigcap A =\mathfrak{p}$ and the above diagram commutes.
		
		\item $C$ is the integral closure of $B/\mathfrak{r}$ in its field of fractions.
		
		\item $B$ is finite over $A$ with $[ K(B):K(A)]= [K(C):{K(A/\mathfrak{p})}]$.
	\end{enumerate}
\end{lem}
\begin{proof}
	Since $C$ is a finitely generated $A/\mathfrak{p}$-algebra, writing $C=A/\mathfrak{p}[y_1,y_2,\ldots,y_k]$, we consider the inclusions of fields 
	$$K (A/\mathfrak{p}) \subseteq K(A/\mathfrak{p})(y_1) \subseteq  K(A/\mathfrak{p})(y_1,y_2) \subseteq \cdots \subseteq K(A/\mathfrak{p})(y_1,y_2,\ldots,y_k)=K(C).$$
	We will prove the lemma inductively. If all the inclusions above are equalities, then there is nothing to prove as $K (A/\mathfrak{p})=K(C)$. So we assume at least one inclusion is strict. After re-indexing we may assume the inclusion $K (A/\mathfrak{p}) \subset K(A/\mathfrak{p})(y_1)$ is strict. 
	
	Let $C_1=A/\mathfrak{p}[y]$ be the sub-ring of $C$ generated by $y=y_1$ over $A/\mathfrak{p}$. Let $$\overline{\alpha}(X)=X^n+\overline{a_1}X^{n-1}+\ldots+ \overline{a_n} \in K(A/\mathfrak{p})[X]$$ be the minimal polynomial of $y$. Since all coefficients $\overline{a_i} \in K(A/\mathfrak{p})$, there is an element $\overline{f}\in A/\mathfrak{p} $ such that $\overline{a_i} \in (A/\mathfrak{p})_{\overline{f}}$. For each $i$, we pick $a_i \in A$, and pick $f \in  A\backslash\mathfrak{p}$ so that $\overline{a_i}$'s and $\overline{f}$ are the reductions of $a_i$'s and $f$ mod $\mathfrak{p}$ respectively. We therefore construct a polynomial $\alpha(X)=X^n+a_1X^{n-1}+\ldots + a_n \in A_f[X]$, which gives a surjetice homomorphism $\phi:A_f[X] \to A_f[w]:=A_f[X]/(\alpha)$ and then induces a surjective homomorphism $\rho:A_f[w] \to C_{1,\overline{f}}$. 
	
	Because $C_1$ is a domain and so is the localisation of $C_{1}$, the kernel of $\rho$ is a prime ideal, denote by $\mathfrak{a}$. It is obvious that $\mathfrak{a} \bigcap A_f=\mathfrak{p}A_f$. Let $\mathfrak{q}_{\min} \subseteq  \mathfrak{a}$ be a minimal prime ideal, $\mathfrak{q}_{\min} \bigcap A_f= \mathfrak{p}_{\min}$, and $\mathfrak{Q}_{\min}=\phi^{-1} \mathfrak{q}_{\min}$. By definition $\mathfrak{Q}_{\min}$ is a minimal prime ideal containing $\alpha$, and by assumption $A_f[X]$ is Noetherian. So, by Krull's Hauptidealsatiz \cite[Corollary 11.17]{am}, we deduce that the height $\mathrm{ht}(\mathfrak{Q}_{\min})=1$. On the other hand, $\mathfrak{Q}_{\min}$ contracts to $\mathfrak{p}_{\min}$ in $A_f$. We claim $\alpha \notin \mathfrak{p}_{\min}A_f[X]$. In fact, if we assume the opposite, then we deduce that, after reduction mod $\mathfrak{p}A_f$, we have $\overline{\alpha}=0 \in	A_f/\mathfrak{p}A_f [X]=(A/\mathfrak{p})_{\overline{f}}[X]$ which is a contradiction. Therefore we have $\mathfrak{Q} \supsetneq \mathfrak{p}_{\min}A_f[X]$, hence $\mathrm{ht}(\mathfrak{Q}_{\min})=\mathrm{ht}(\mathfrak{p}_{\min})+1$ which in turn implies that $\mathfrak{p}_{\min}=(0)$.  
	
	Now we have a surjective homomorphism $q_U: U_1:=A_f[w]/\mathfrak{q}_{\min} \to C_{1,\overline{f}}$. Moreover, since $\mathfrak{q}_{\min} \bigcap A_f=(0)$, it in turn produces a finite domain extension $A_f \subset U_1$. We denote by $x \in U_1$ the image of $w$. Because $A_f$ is integrally closed and $x$ is integral over $A_f$, by \cite[Proposition 5.15]{am}, it turns out that the minimal polynomial of $x$ over $K(A)$, denoted by $\alpha'(X)=X^m+a_1' X^{m-1}+\ldots + a_m'$, has all its coefficients $a_i' \in A_f$ and $m$ divides $n$. Since the minimal polynomial of an algebraic element over a field is unique, by $q_U(\alpha'(x))=0$, one can easily obtain $\overline{a_i'}=\overline{a_i}$ and $m=n$, where $\overline{a_i'}=q_U(a_i)$. It follows that $U_1=A_f[X]/(\alpha')$ and $[K(U_1):K(A)]=[K(C_1):{K(A/\mathfrak{p})}]$.
	
	Let $B_1$ be the integral closure of $A$ in $U_1$, and consider the extension of domains $A \subset B_1$. Since $U_1$ is finite over $A_f$, by \cite[Proposition 5.12]{am} we deduce $U_1=B_{1,f}$. Furthermore, if we denote by $C_1'$ the image of $B_1$ in $C_{1,\overline{f}}$ through $q_U$, then we have $C_{1,\overline{f}}'=C_{1,\overline{f}}$. Therefore, we have $K(B_1)=K(U_1)$ and $K(C_1')=K(C_1)$. We note that every element of $C_1'$ is integral over $A/\mathfrak{p}$, hence integral over $C$ which in turn implies that it is an element of $C$. So, we have $C_1' \subset C$. Let $q_1: B_1 \to C_1'$ be the surjective homomorphism induced by $q_U$. We obtain a commutative diagram.
	$$
	\xymatrix{
		B_1  \ar@{->>}[r]^{q_1}   &  C_1'  &\\
		A \ar@{^(->}[u]^{}\ar@{->>}[r]^{p} &    A/\mathfrak{p} \ar@{^(->}[u]^{} } 
	$$
	
	Next we let $\overline{B_1}$ be the integral closure of $B_1$ in its field of fractions. Since $A$ is a Nagata ring, by \cite[Corollary 5.4]{am}, we see $\overline{B_1}$ is the integral closure of $A$ in $K(B_1)$, and therefore it is finite over $A$. Moreover, $\overline{B_1}$ is also a Nagata ring. 
	By \cite[Theorem 5.11]{am}, there exists a prime ideal $\overline{\mathfrak{r}_1}$ such that $\overline{\mathfrak{r}_1} \bigcap B_1=\mathfrak{r}_1$ where $\mathfrak{r}_1$ is the kernel ideal of $q_1$. Letting $\overline{C_1}=\overline{B_1}/\overline{\mathfrak{r}_1}$, we consider the commutative diagram 
	$$
	\xymatrix{
		\overline{B_1} \ar@{->>}[r]^{\overline{q_1}}    &   \overline{C_1} 	\\
		B_1 \ar@{^(->}[u]^{} \ar@{->>}[r]^{q_1}   &  C_1' \ar@{^(->}[u]^{}  &\\
		A \ar@{^(->}[u]^{}\ar@{->>}[r]^{p} &    A/\mathfrak{p} \ar@{^(->}[u]^{}} 
	$$ 
	of which all horizontal arrows are surjective homomorphisms. Suppose one of the following conditions holds:
	\begin{enumerate}
		\item The residue field $\kappa(\mathfrak{p})$ is perfect.
		
		\item $A$ is regular at $\mathfrak{p}$.
	\end{enumerate}
	By Lemma \ref{lem-com-alg-1} and Remark \ref{rem-com-alg-1}, we deduce
	\begin{align*}
	[K(B_1):K(A)] &=[K(\overline{B_1}):K(A)]\\
	&\ge [K(\overline{C_1}):{K(A/\mathfrak{p})}] =[K(\overline{C_1}):K(C_1')][K(C_1'):{K(A/\mathfrak{p})}].
	\end{align*}
	Hence we achieve the equation $[K(\overline{C_1}):K(C_1')]=1$ and $$[K(\overline{B_1}):K(A)]=[K(\overline{C_1}):{K(A/\mathfrak{p})}].
	$$
	By assumption $C$ is integrally closed. Hence the integral closure of $C_1$ in its field of fractions lies in $C$ which in turn implies that $\overline{C_1}$ is a sub-ring of $C$. Also note that, in the first case, $\kappa(\overline{\mathfrak{r}_1})$ is perfect; in the second case, by Lemma \ref{lem-com-alg-3}, $\overline{B_1}$ is regular at $\overline{\mathfrak{r}_1}$. Therefore, we may continue the procedure above for $\overline{B_1},\overline{C_1}$ instead of $A, A/\mathfrak{p}$. Because $[K(C):K(\overline{C_1})]< [K(C):{K(A/\mathfrak{p})}]$, after finite steps we obtain the required integrally closed domain $B$.
\end{proof}

\begin{proof}[Proof of Theorem \ref{thm-extend-finite-cover}]
	Let $U \subseteq X$ be an affine open subset such that $U \bigcap S \ne \emptyset$, and let $S_U,\widetilde{S}_U$ be affine open subsets $S \bigcap U, \widetilde{S} \bigcap \gamma^{-1}S_U$ of $S,\widetilde{S}$ respectively. By the previous lemma, there exists a finite morphism $\rho_U:\widetilde{U} \to U$ with a closed subvariety $\widehat{S}_U \subset \widetilde{U}$ which induces a normalisation $\nu_U: \widetilde{S}_U \to \widehat{S}_U $. Moreover, we have $\deg \rho_U=\deg \gamma$. 
	
	Let $X \hookrightarrow X^c$ be an open immersion to a complete normal variety and let $S^c$ be the Zariski closure of $S$ in $X^c$. Thus, there exists an open immersion $\widetilde{S} \hookrightarrow \widetilde{S}^c$ such that the induced map $\gamma^c:\widetilde{S}^c \dashrightarrow S^c$ is generically finite dominant. Replacing $\gamma^c$ by the Stein factorisation we can assume it is still finite. Replacing $X,S,\widetilde{S}$ with $X^c,S^c,\widetilde{S}^c$, we assume they are complete and consider the following commutative diagram:
	$$
	\xymatrix{
		\widetilde{S} \ar[dr]_{\gamma} \ar@{-->}[r]^{}& \widehat{S}_U \ar[d]_{} \ar@{^(->}[r]_{} &  \widetilde{U}\ar[d]^{\rho_U}\  &\\
		& S \ar@{^(->}[r]^{} &    X } 
	$$ 
	Let $\widetilde{U} \hookrightarrow \widetilde{X}$ be an open immersion to a complete normal variety such that the induced map $\rho: \widetilde{X} \dashrightarrow X$ is a proper surjective morphism and let $\widehat{S}$ be the Zariski closure of $\widehat{S}_U$ in $\widetilde{X}$. Replacing $\widetilde{X}$ by the Stein factorisation we can assume further that $\rho$ is a finite morphism. Since $\widehat{S} \to S$ is a finite morphism and the induced map $\nu: \widetilde{S} \dashrightarrow \widehat{S}$ is birational, we conclude that $\nu$ the normalisation.
\end{proof}

\begin{rem}
	One can arrange a scheme-theoretic argument so that Theorem \ref{thm-extend-finite-cover} can be generalised to a quasi-compact and quasi-separated normal Nagata scheme $X$ with an integral closed subscheme $S$, instead of varieties. Indeed, with notation above, by Nagata compactification \cite[\href{https://stacks.math.columbia.edu/tag/0F41}{Theorem 0F41}]{stacks-project}, we may compactify $\rho_U$ to a proper morphism $\rho: \widetilde{X} \to X$. Apply the normalisation \cite[\href{https://stacks.math.columbia.edu/tag/035L}{Lemma 035L}]{stacks-project} and the Stein factorisation \cite[\href{https://stacks.math.columbia.edu/tag/03H0}{Theorem 03H0}]{stacks-project} to reduce to the case when $\widetilde{X}$ is normal and $\rho$ is finite, hence the conclusion. 
\end{rem}

\section{Weak semi-stable reductions}\label{subsec-ssred}
In this paper we apply theorems of weak toroidal reduction and weak semi-stable reduction developed by Abramovich, Denef and Karu \cite{adk}\cite{ak}. See also \cite[Section 4]{ambro1}. We will require an extra condition on the base variety. Note that the extra condition will cause no troubles, so the proofs from \cite{adk}\cite{ak} are still sufficient. For the reader's convenience, we give a sketch of proof. For detailed arguments we refer to \cite{adk}\cite{ak}.

\begin{note}\label{note-toroidal}
	A pair $(X,\Delta)$ with a reduced divisor $\Delta$ is a \emph{(strict) toroidal variety} if $X \setminus \Delta \subset X$ is a \emph{(strict) toroidal embedding} (\cite[Definition 1.2]{ak}). All toroidal varieties in this paper are assumed to be strict. A morphism $f:(X,\Delta) \to (Y ,\Delta_Y)$ between toroidal varieties is \emph{toroidal} if it is toroidal between toroidal embeddings (\cite[Definition 1.3]{ak}).
\end{note}

We collect some basic properties from toroidal geometry for the reader's convenience:

\begin{lem}\label{lem-toroidal-var}
	Let $(X,\Delta)$ be a toroidal variety. Then we have:
	\begin{enumerate}
		\item $X$ is normal and Cohen-Macaulay, and $(X,\Delta)$ is lc.
		
		\item If $X$ is smooth, then $(X,\Delta)$ is log smooth. Conversely, any log smooth pair is a toroidal variety.
		
		\item If $(X,\Delta)$ is quasi-smooth, then $X$ has only abelian quotient singularities. In particular, $X$ is $\Q$-factorial klt.
		
		\item $(X,\Delta)$ defines a natural stratification, and lc centres are exactly the closures of strata, hence an lc centre $S$ inherits the toroidal structure. Moreover, $S$ is (quasi-)smooth if $X$ (quasi-)smooth, 
	\end{enumerate}
\end{lem}

\begin{lem}\label{lem-toroidal}
Let $f:(X,\Delta) \to (Y ,\Delta_Y)$ be a toroidal morphism. Then we have:
   \begin{enumerate}
   	\item The morphism $f$ maps strata to strata. In particular, $\Delta^v=f^{-1}\Delta_{Y}$ where $\Delta^v$ denotes the vertical part of $\Delta$.
   	
   	\item Suppose $X$ is smooth. Then, for any reduced divisor $\Delta'$ with $\Delta-\Delta' \ge 0$ horizontal, the morphism $f:(X,\Delta') \to (Y ,\Delta_Y)$ is toroidal. Moreover, given a horizontal stratum $S$, the restriction $f|_S$ is again toroidal.  In particular, $f$ is log smooth over $Y \setminus \Delta_Y$. Conversely, any log smooth morphism is toroidal.
   	
   	\item Suppose $Y$ is quasi-projective and $H$ is a general hyperplane. Then, $f:(X,\Delta+f^{-1}(H)) \to (Y ,\Delta_Y+H)$ is toroidal.
   \end{enumerate}
\end{lem}

The following theorem is \cite[Theorem 1.1]{adk}\cite[Theorem 2.1]{ak}. Note that Condition (3) below is similar to \cite[Theorem 1.1, Condition (4)]{adk} but with an extra requirement on $Y$.

\begin{thm}[Weak toroidal reduction]\label{thm-fib-resolution}
	Let $f:X \to Y$ be a dominant morphism of varieties and $Z \subset X,Z_Y\subset Y$ be proper closed subsets. Then there exist proper birational morphisms $\pi,\phi$ and a commutative diagram
	$$\xymatrix{
		(\overline{X},\overline{\Delta}) \ar[d]_{\overline{f}} \ar[r]^{\pi}   &  X \supset  Z\ar[d]^{f}  \\
		(\overline{Y},\Delta_{\overline{Y}}) \ar[r]^{\phi} &    Y  \supset  Z_Y 
	} 
	$$  
	satisfying the following conditions:
	\begin{enumerate}
		\item $(\overline{X},\overline{\Delta}),(\overline{Y},\Delta_{\overline{Y}})$ are quasi-projective.

		\item $(\overline{X},\overline{\Delta}),(\overline{Y},\Delta_{\overline{Y}})$ are smooth toroidal varieties and $\overline{f}$ is a toroidal morphism.
	
		\item $\pi^{-1}Z \bigcup \ex(\pi) \subseteq \overline{\Delta}$ and $\phi^{-1}Z_Y \bigcup \ex(\phi) \subseteq \Delta_{\overline{Y}}$, where $\ex(\pi),\ex(\phi)$ denote the exceptional loci. 
		
		\item[(3$\dag$)] $\pi^{-1}Z,\phi^{-1}Z_Y,\ex(\phi)$ are divisors.
	\end{enumerate}	
    Furthermore, if $Z$ is vertical, then $(\overline{X},\overline{\Delta}^v)$ also satisfies the conditions above, except that $\ex(\pi) \subseteq \overline{\Delta}$, where $\overline{\Delta}^v$ denotes the vertical part of $\overline{\Delta}$.
\end{thm}
\begin{proof}[Sketch of Proof]
	Replacing $X,Y$ we can assume they are projective and $Z,Z_Y$ are the supports of Cartier divisors (\cite[3.1]{adk}). 
	
	{\noindent \textbf{Step 1.}} We prove by induction on the relative dimension $n$ of $f$. If $n=0$, then we may construct a finite morphism $\overline{f}:(\overline{X},\overline{\Delta}) \to (\overline{Y},\Delta_{\overline{Y}})$ ramified over the snc divisor $\Delta_{\overline{Y}}$ which satisfies all the conditions, except that $(\overline{X},\overline{\Delta})$ is not necessarily log smooth (see \cite[3.4.2]{adk}). It is toroidal by Abhyankar's lemma \cite[Lemma 3.3]{adk}. Replacing $(\overline{X},\overline{\Delta})$ with a toroidal resolution we immediately obtain the conclusion.
	
	{\noindent \textbf{Step 2.}} By induction we suppose Theorem \ref{thm-fib-resolution} holds in relative dimension $n-1$. Replacing $X,Y,Z,Z_Y$ we can factorize $f:X \overset{g}{\to} P \overset{h}{\to} Y$ with relative dimension $n-1$ of $h$ (see \cite[3.5]{adk}) and assume $f^{-1}Z_Y \subset Z $. 
	Now construct a commutative diagram
	$$
	\xymatrix{
		\widehat{X} \ar[d]_{\widehat{g}} \ar[r]^{\mu}	& \widehat{X}/G \ar[d]_{g'} \ar[r]^{\pi}   &  X\ar[d]_{g}  \supset Z \ar@/^/[dd]^{f }\\
		\widehat{P} \ar[r]^{\gamma}	& \widehat{P}/G \ar[r]^{\phi} &    P\ar[d]_{h} \\
		&  & Y \supset Z_Y } 
	$$  
	such that (\cite[Theorem 2.4, Remark 2.5]{dejong}):
	\begin{enumerate}
		\item $\mu':=\pi \circ \mu, \gamma':=\phi \circ \gamma$ are Galois alterations with Galois group $G$.
		
		\item $\widehat{g}:\widehat{X} \to \widehat{P}$ is a nodal family of curves with irreducible smooth general fibres.
		
		\item There are finitely many disjoint sections $\sigma_i: \widehat{P} \to \widehat{X}$ into the smooth locus of  of $\widehat{g}$ such that $G$ permutes the sections $\widehat{D}_i:=\sigma_i(\widehat{P})$.
		
		\item $\widehat{Z}=\mu'^{-1}Z \subset (\bigcup_i \widehat{D}_i)\bigcup \widehat{g}^{-1}(\widehat{D})$ for some proper closed subset $\widehat{D} \subset \widehat{P}$.
	\end{enumerate}
	Note that $\ex(\pi)$ is vertical over $\widehat{P}/G $, hence $\mu^{-1}\ex(\pi)$ is vertical over $\widehat{P}$. Replacing $Z$ with $\pi^{-1}Z \bigcup \ex(\pi)$ and $X,P,\widehat{D}$ accordingly we assume $X=\widehat{X}/G, P=\widehat{P}/G$. For the moment, $Z=Z_1\bigcup Z_2$ where $Z_1$ is the support of a Cartier divisor and $g(Z_2) \subset \gamma(\widehat{D})$. 
	
	{\noindent \textbf{Step 3.}} Write $Z_P:=\gamma(\widehat{D}) \bigcup \{\text{the loci over which $Z, P$ or $X$ are not smooth}\}$. 
	Applying the inductive assumption to $h:(P,Z_P) \to (Y,Z_Y)$, we construct a commutative diagram:
	$$
	\xymatrix{
		\widetilde{X}\ar[dr]_{\overline{\mu}}\ar[dd]_{\widetilde{g}}\ar[rr]^{\widetilde{\pi}} && \widehat{X} \ar[dr]^{\mu}\ar[dd]_{}\\
		& \overline{X}\ar[dd]^{ }   \ar[rr]^{\overline{\pi}} & & X \ar[dd]^{g} \\
		\widetilde{P}\ar[dr]_{\overline{\gamma}}\ar[rr]^{\ } && \widehat{P} \ar[dr]^{\gamma} \\
		&(\overline{P},\Delta_{\overline{P}}) \ar[d]_{\overline{h}} \ar[rr]^{\overline{\phi}} &  &  P\ar[d]^{h}\  &\\
		&(\overline{Y},\Delta_{\overline{Y}}) \ar[rr]^{} &  &  Y \\
	} 
	$$
	where $\overline{h}:(\overline{P},\Delta_{\overline{P}}) \to (\overline{Y},\Delta_{\overline{Y}})$ satisfies all the conditions. Since $g$ is equidimensional, if we write $\overline{X},\widetilde{X},\widetilde{P},\widetilde{D}_i$ induced by base change, $\overline{Z}:=\overline{\pi}^{-1}Z \bigcup \ex(\overline{\pi})$, then by inductive assumption we deduce $\overline{g}(\ex(\overline{\pi})) \subset \Delta_{\overline{P}}$ which in turn implies that $\overline{Z}=\overline{Z}_1\bigcup \overline{Z}_2$ where $\overline{Z}_1= \overline{\pi}^{-1} Z_1$ and $\overline{g}(\overline{Z}_2) \subset \Delta_{\overline{P}}$. 
	
	By Abhyankar's lemma again, we see $(\widetilde{P},\Delta_{\widetilde{P}}=\overline{\gamma}^{-1}\Delta_{\overline{P}})$ is toroidal, hence $(\widetilde{X},\widetilde{\Delta}=(\gamma \circ \widetilde{g})^{-1}\Delta_{\overline{P}} + \sum_i \widetilde{D}_i)$ is toroidal (\cite[1.3]{adj}). Since $\overline{g}(\overline{Z}_2) \subset \Delta_{\overline{P}}$, we deduce $\overline{\mu}^{-1} \overline{Z} \subset \widetilde{\Delta}$. 
	
	{\noindent \textbf{Step 4.}} Finally let us modify $\widetilde{X}$ to a $G$-equivariant toroidal variety and then $\overline{X}$ accordingly (\cite[3.9]{adk}). Writing $U_{\widetilde{P}}=\widetilde{P} \backslash \Delta_{\widetilde{P}}$, because the $G$-equivariant resolution $b:(\widetilde{X}',\widetilde{\Delta}') \to (\widetilde{X},\widetilde{\Delta})$ preserves $\widetilde{U}$, we deduce $\ex(b) \subset \widetilde{\Delta}$ (\cite[3.9.1, 3.9.5, 3.9.6]{adk}). Letting $(\overline{X}',\overline{\Delta}')=(\widetilde{X}'/G,\widetilde{\Delta}'/G)$ and $\psi:\overline{X}' \to \overline{X}$ be the induced modification, we see $\ex(\psi) \subset \overline{\Delta}'$ and $(\overline{X}',\overline{\Delta}') \to (\overline{P},\Delta_{\overline{P}})$ is toroidal (\cite[2.3]{adk}). Replacing $(\overline{X},\overline{\Delta})$ with a toroidal resolution of $(\overline{X}',\overline{\Delta}')$ we complete the proof.
\end{proof}

\begin{rem}\label{rem-toroidal-res}
	Since we suppose both $Z$ and $Z_Y$ are the supports of Cartier divisors, it follows that $\pi^{-1}Z$ and $\pi^{-1}Z_Y$ are snc divisors by Krull's Hauptidealsatz. By log resolution we can assume $\ex(\phi)$ is also a divisor. This proves (3$\dag$).
	
	However, in general we have no idea how to resolve $\ex(\pi)$ to an snc divisor unless components of $\ex(\pi)$ are toroidal strata. Hence, we are interested if one can further require the exceptional locus $\ex(\pi)$ to be snc.
\end{rem}


\begin{thm}[Equidimensional reduction]\label{thm-equidimension}
	Let $f: X\to Y$ be a dominant morphism of normal varieties and $Z \subset X,Z_Y\subset Y$ be proper closed subsets. Then there exists a commutative diagram
	$$
	\xymatrix{
		(X_1,\Delta_1) \ar[d]_{f_1} \ar[r]^{\pi}   &  X\ar[d]^{f}\  &\\
		(Y_1,\Delta_{Y_1}) \ar[r]^{\phi} &    Y } 
	$$  
	satisfying the following conditions:
	\begin{enumerate}
		\item $\pi$ and $\phi$ are birational morphisms from quasi-projective varieties. 
		
		\item $f_1:(X_1,\Delta_1) \to (Y_1 ,\Delta_{Y_1})$ is toroidal from a quasi-smooth toroidal variety to a smooth variety, and $f_1^{-1}(Y_1 \setminus \Delta_{Y_1})$ is smooth.
		
		\item $f_1$ is flat. In particular, all the fibers of $f_1$ have the same dimension. 
		
		\item $\pi^{-1}Z \bigcup \ex(\pi)  \subseteq \Delta_1$ and $\phi^{-1}Z_Y \bigcup \ex(\phi) \subseteq \Delta_{Y_1}$, where $\ex(\pi),\ex(\phi)$ denote the exceptional loci. 
	\end{enumerate}
	Furthermore, if $Z$ is vertical$/Y$, then $(X_1,\Delta_1^v)$ also satisfies the conditions listed above, except that $\ex(\pi) \subseteq \Delta_1$, where $\Delta_1^v$ denotes the vertical part of $\Delta_1$.
\end{thm}
\begin{proof}[Sketch of Proof]
	By Theorem \ref{thm-fib-resolution} there exist modifications $\overline{\pi},\overline{\phi}$ and a toroidal morphism $\overline{f}:(\overline{X},\overline{\Delta}) \to (\overline{Y},\Delta_{\overline{Y}})$ satisfying the conditions listed in Theorem \ref{thm-fib-resolution}. By \cite[Proposition 4.4]{ak} there exist projective toroidal modifications $\pi_1:X_1 \to \overline{X}$ and $\phi_1: Y_1 \to \overline{Y}$ such that the induced map $f_1$ satisfies all the conditions except that $X_1$ is quasi-smooth. Moreover, by \cite[Remark 4.5]{ak} and replacing $X_1$, we can assume $X_1$ is quasi-smooth. 
\end{proof}

\begin{thm}[Weak semi-stable reduction]\label{thm-ss-reduction}
	With the notation of Theorem \ref{thm-equidimension}, if we suppose further that $f$ has the geometrically connected generic fibre, then, by adding appropriate general hyperplanes to $\Delta_{Y_1}$ and replacing
	$\Delta_1$ accordingly (see Lemma \ref{lem-toroidal}(3)), there exists a commutative diagram
	$$
	\xymatrix{
		(X_2,\Delta_2) \ar[d]_{f_2} \ar[r]^{\mu}	& (X_1,\Delta_1) \ar[d]_{f_1} \ar[r]^{\pi}   &  X\ar[d]^{f}\  &\\
		(Y_2,\Delta_{Y_2}) \ar[r]^{\gamma}	& (Y_1,\Delta_{Y_1}) \ar[r]^{\phi} &    Y } 
	$$  
	satisfying the following conditions:
		\begin{enumerate}
		\item $\gamma$ is a Galois finite toroidal morphism, and $f_2$ is induced by base change. 
		
		\item $f_2:(X_2,\Delta) \to (Y_2 ,\Delta_{Y_2})$ is toroidal from a quasi-smooth toroidal variety to a smooth variety, and $f_2^{-1}(Y_2 \setminus \Delta_{Y_2})$ is smooth.
		
		\item All the fibers of $f_2$ are reduced, and $f_2$ is flat.
		
		\item $\mu^{-1}(\pi^{-1}Z \bigcup \ex(\pi)) \bigcup R \subseteq \Delta_2$ and $\gamma^{-1}(\phi^{-1}Z_Y \bigcup \ex(\phi)) \bigcup R_Y \subseteq \Delta_{Y_2}$, where $R,R_Y$ are the ramification sets of $\mu,\gamma$ respectively.
	\end{enumerate}
	Furthermore, if $Z$ is vertical$/Y$, then $(X_2,\Delta_2^v)$ also satisfies the conditions listed above, except that $\mu^{-1}\ex(\pi) \subseteq \Delta_2$, where $\Delta_2^v$ denotes the vertical part of $\Delta_2$.
\end{thm}
\begin{proof}[Sketch of Proof]
	By \cite[Proposition 5.1]{ak}, adding appropriate general hyperplanes to $\Delta_{Y_1}$, there exists a finite toroidal morphism $\gamma : Y_2 \to Y_1$ between smooth varieties so that the induced morphism $f_2:( X_2,\Delta_2) \to  (Y_2,\Delta_{Y_2})$ is flat toroidal with reduced fibres. By \cite[Proposition 5.10]{ak} one obtains that $(X_2,\Delta_2)$ is quasi-smooth.
\end{proof}

\begin{rem}
	The connectedness condition on fibres can be removed. But note that in this case, $X_1 \times_{Y_1} Y_2$ possibly have multiple main components, hence $X_2$ is a disjoint union of finitely many varieties. See \cite[Theorem 4.3]{ambro1}\cite[Lemma 5.6]{ak}.
\end{rem}


\end{document}